\documentclass[reqno, 12 pt]{amsart}
\usepackage[utf8]{inputenc}

\usepackage{amssymb} 
\usepackage{mathrsfs}  
\usepackage{amsmath}  
\usepackage{euscript} 
\usepackage{amsthm} 
\usepackage{dsfont}  
\usepackage{amsfonts}      
\usepackage{marginnote}
\usepackage{latexsym} 
\usepackage{amsopn} 
\usepackage{amscd} 
\usepackage{caption} 
\usepackage{aurical}  
\usepackage[english]{babel}
\usepackage{esint} 
\usepackage{graphicx}
\usepackage[dvipsnames]{xcolor}
\usepackage{multicol}
\usepackage{nomencl}
\makeglossary
\makenomenclature  
\usepackage{refcount}
\usepackage{float}
\usepackage{times}
\usepackage[titletoc]{appendix}
\usepackage{accents}
\usepackage{scalerel,stackengine}
\usepackage{tikz-cd}
\usepackage{footmisc}
\usepackage{setspace}
\usepackage{verbatim}

\usepackage[
paper=a4paper,
textwidth=15.5cm,
textheight=22.5cm,
centering,
heightrounded 
]{geometry}

\usepackage{hyperref}
\usepackage{mathtools}
\mathtoolsset{showonlyrefs}

\stackMath
\newcommand\rwhat[1]{%
	\savestack{\tmpbox}{\stretchto{%
			\scaleto{%
				\scalerel*[\widthof{\ensuremath{#1}}]{\kern-.6pt\bigwedge\kern-.6pt}%
				{\rule[-\textheight/2]{1ex}{\textheight}}%
			}{\textheight}%
		}{0.5ex}}%
	\stackon[1pt]{#1}{\tmpbox}%
}
\newlength{\dhatheight}

\newcommand{\eps}{\varepsilon}

\newcommand{\sgn}{\operatorname{sgn}}

\newtheorem{thm}{Theorem}[section]
\newtheorem*{thm*}{Theorem}
\newtheorem{prop}[thm]{Proposition}
\newtheorem{lemma}[thm]{Lemma}

\newtheorem{cor}[thm]{Corollary}
\newtheorem*{cor*}{Corollary}
\newtheorem{rmk}[thm]{Remark}
\newtheorem{ex}[thm]{Example}
\newtheorem{ex*}{Example}

\newtheorem{defn}[thm]{Definition}
\newtheorem{thmA}{Theorem}

\newtheorem{thmB}{Theorem}

\newtheorem{thmC}{Theorem}

\newtheorem{thmD}{Theorem}

\newtheorem{thmBO}{Theorem}

\newtheorem{notation}[thm]{Notation}

\numberwithin{equation}{section}

\newcommand{\om}{{\omega}}

\newcommand{\C}{{\mathbb C}}

\newcommand{\N}{{\mathbb N}}
\newcommand{\Q}{{\mathbb Q}}
\newcommand{\R}{{\mathbb R}}

\newcommand{\T}{{\mathbb T}}
\newcommand{\Z}{{\mathbb Z}}

\newcommand{\tw}{{\mathtt{w}}}

\usepackage{bm}

\newcommand{\im}{{\rm i}}
\newcommand{\jap}[1]{\langle #1 \rangle}

\newcommand{\nnorm}[1]{{\left\vert\kern-0.25ex\left\vert\kern-0.25ex\left\vert #1 
		\right\vert\kern-0.25ex\right\vert\kern-0.25ex\right\vert}}

\usepackage{framed,enumitem}

\makeatletter

\def\l@subsection{\@tocline{2}{0pt}{2.5pc}{5pc}{}}
\def\l@subsubsection{\@tocline{3}{0pt}{4.5pc}{5pc}{}}
\renewcommand\tocchapter[3]{%
	\indentlabel{\@ifnotempty{#2}{\ignorespaces#2.\quad}}#3%
}

\newcounter{paraga}[subsection]
\renewcommand{\theparaga}{{\bf\arabic{paraga}.}}
\newcommand{\paraga}{\medskip \addtocounter{paraga}{1}
\noindent{\theparaga\ } }

\begin{document} 
	
	\author{Lorenzo Baroni}
	\address{Università degli Studi Roma Tre}
	\email{lorenzo.baroni@uniroma3.it}

\begin{abstract}
	This article is concerned with Kronecker flows on the infinite torus. The work is partly motivated by the fact that many Hamiltonian PDEs and systems on infinite lattices admit invariant tori, of possibly infinite dimension, on which the dynamics is linearizable. Finite-dimensional Kronecker flows are well understood: the dynamics can be reduced to a non-resonant flow on a subtorus, which is equivalent to being topologically transitive, to minimality, and to unique ergodicity in the projection. We prove that these properties still hold when the dimension of the torus is infinite if and only if the integer (finite) linear combinations of the frequencies form a free abelian group. Next, we construct a class of orbits whose closure is locally homeomorphic to the product of a ball and a Cantor set, extending a recent result by Sakbaev and Volovich. We also show that the Benjamin-Ono equation admits this type of solutions. Finally, we prove the equivalence between a classification problem for Kronecker flows and that for countable abelian groups without torsion.
	\end{abstract}

\title{Kronecker Flow on the Infinite Torus}
\maketitle
\tableofcontents

\section{Introduction and Main Results}
\subsection{Context} Kronecker flows play a central role in Hamiltonian dynamics as they describe the almost-periodic motion on compact, regular common level sets of complete systems of Poisson-commuting first integrals. The Kronecker flow on the regular sets is defined by a particular Hamiltonian which admits these first integrals. In the finite-dimensional case, KAM theory establishes the persistence of most of such almost-periodic solutions under small perturbations. In general, this is not the case in infinite-dimensional Hamiltonian systems. Still, the search for almost-periodic solutions is an active and challenging topic in the study of global solutions for evolution PDEs. Therefore, understanding Kronecker flows on infinite tori is crucial for gaining insights into the dynamical and topological properties of almost-periodic solutions of Hamiltonian PDEs and systems on infinite lattices.

When the dimension of the torus is finite, Kronecker flows satisfy the following properties.

\begin{enumerate}[label=\textbf{\Alph*}, ref=\Alph*]
	\item \label{prop:nonres} The (non-resonance) condition that the frequencies are rationally independent is equivalent to topological transitivity, to minimality, and to unique ergodicity (see \cite{lazutkin} for a modern exposition).
	\item \label{prop:kronecker} Any Kronecker flow is topologically conjugate by an automorphism of the torus to another Kronecker flow having the property that the non-zero frequencies are rationally independent (see, for e.g., \cite{fasano}). Consequently, by Property \ref{prop:nonres}, the closure of any orbit is still a torus.
\end{enumerate}

Preliminaries studies of the infinite-dimensional analogue are due to Jessen \cite{Jess}, who essentially proves a special case of Birkhoff's Ergodic Theorem. The ergodic properties of Kronecker flows on infinite tori with the product topology are also studied in \cite{kozov} and \cite{SakVol}. In the same paper \cite{SakVol}, Sakbaev and Volovich also show that there is a new type of trajectories that is absent in the finite-dimensional case, namely, orbits whose closure is not a torus. Therefore, the properties of Kronecker flow in the infinite-dimensional setting are much more involved than in the finite-dimensional case.

\subsection{Purposes and Informal Short Summary of the Main Results} In this paper we consider Kronecker flows on (countably) infinite tori endowed with the product topology. This particular choice of the topology, which coincides with the one considered in \cite{Jess}, \cite{kozov} and \cite{SakVol}, is motivated in Section \ref{motivations} and it is related to the topology of the functional spaces typically considered in the study of Hamiltonian PDEs. We refer to the next section for precise definitions. 

\paraga The first question addressed in this work can be informally rephrased as follows.

\begin{center}
\medskip
\textit{Does there exist a class of Kronecker flows on infinite tori which share the same properties \footnote{In the sense of Properties \ref{prop:nonres} and \ref{prop:kronecker}.} as in the finite-dimensional case?}
\medskip
\end{center}

We prove that this question has positive answer (see Theorem \ref{A} and \ref{B}). Namely, we identify this special class as the class of Kronecker flows whose frequency module — the $\mathbb{Z}$-module of all integer (finite) linear combinations of the frequencies — is a free abelian group. In this light, Properties \ref{prop:nonres} and \ref{prop:kronecker} of the finite-dimensional case follow form the structure theorem for finitely generated abelian groups without torsion, which ensures that the frequency module is always free in the case of Kronecker flows with only finitely many non-zero frequencies (see Remark \ref{remfinfree}).

\paraga It is not difficult to provide examples of Kronecker flows on infinite tori whose associated frequency module is not free. This remark motivates the next question tackled in this paper.

\begin{center}
	\medskip
	\textit{In the case of Kronecker flows with non-free frequency module, what is the topological structure of the closure of the orbits?}
	\medskip
\end{center}

In Section \ref{ratflow} and \ref{orbit-dual} we construct a class of Kronecker flows with the property that the closure of any orbit is a product of circles and solenoids\footnote{Solenoids are compact, connected subgroups of the infinite torus. We give the definition in Section \ref{ratflow} and we also study their local topological properties (see Proposition \ref{localcantor}). For an intrinsic definition, see for e.g. \cite{HR}, in which they go under the name of $\textit{\textbf{a}}$\textit{-adic Solenoids}.}. This can be viewed as an extension of the result of Sakbaev and Volovich \cite{SakVol}, who showed that for linear flows on the infinite torus there exist orbits whose closure is not a torus (see Remark \ref{rmksakvol}). Indeed, we exhibit a specific class of such flows and we provide a more refined understanding of the topological structure of such orbits (see Theorem \ref{productcirclesol}).

\paraga At this stage, the following natural question arises.

\begin{center}
\medskip
\emph{Are we able to exhibit examples of Hamiltonian PDEs admitting almost-periodic solutions whose closure is homeomorphic to a product of circles and solenoids?}
\medskip
\end{center}

The Benjamin-Ono equation with periodic boundary condition for a real valued function
\begin{equation}\label{introBO}
u_t = \mathrm{H}(u_{xx}) - 2uu_x, \quad x \in \mathbb{R}/2\pi\mathbb{Z}, \quad t \in \mathbb{R},
\end{equation}
where $\mathrm{H}(\cdot)$ denotes the Hilbert transform
\begin{equation}\label{ht0}
	\mathrm{H}(u)(x,t)=-\frac{\im}{2\pi}\sum_{j \in \Z}\sgn(j)\int_{0}^{2\pi}u(x',t)e^{\im j(x- x')}dx',
	\end{equation}
is an infinite-dimensional, completely integrable (see \cite{BO}), Hamiltonian system that describes one-dimensional, periodic internal waves in deep water approximation. In Theorem \ref{BO} of Section \ref{ssbo}, we show that it admits a family of (non-typical) almost-periodic solutions that are dense on invariant subsets of $L^2(\R/2\pi\Z)$ homeomorphic to the product of a circle and a solenoid. It may be interesting to understand if such family of almost-periodic solutions can be characterized by some physical feature. One can think of some new type of turbulence.

\paraga The last question we answer consists in a classification problem.

\begin{center}
	\medskip
	\textit{Given two Kronecker flows, are we able to determine whether the closures of their orbits are homeomorphic?}
	\medskip
\end{center}

 In Section \ref{stocazzo} we prove that the above question is equivalent to the highly non-trivial classification problem for countably infinite abelian groups without torsion. To date, a complete classification is known only for the rank $1$ case (see \cite{baer}, \cite{beau} and \cite{Fuchs2015}), which corresponds to the class of Kronecker flows whose orbits are dense in a product of circles and solenoids, as we show in Section \ref{orbit-dual}. The key result that bridges such correspondence is Theorem \ref{E}, which we believe is interesting by itself.

\subsection{Setting}\label{setting}
Let us give the fundamental definitions to rigorously introduce the subject of concern and give a brief summary of the main results.

\begin{notation}\label{notations}
	In what follows, we denote by $\mathbb{N} := \{1, 2, \dots\}$ the set of positive integers.
	
	The symbol $\mathbb{T} := \mathbb{R}/2\pi\mathbb{Z}$ denotes the circle group, identified with the additive group of real numbers modulo $2\pi$ and endowed with the quotient topology.
	
	Given an additive abelian group $G$, a subgroup $H \subseteq G$, and an element $g \in G$, we denote by $g + H$ the coset (or equivalence class) of $g$ modulo $H$, defined as:
	\[
	g + H := \{g + h \, : \, h \in H\}.
	\]
	This corresponds to the equivalence relation where $g \sim g'$ if and only if $g - g' \in H$.
\end{notation}

\paraga \textit{Torus.} A torus is a countable direct product of circle groups. More precisely, given a countable set of indices $J$, the torus $\T^J$ is the topological abelian group defined as
\begin{equation}
\mathbb{T}^J := \prod_{j \in J} \mathbb{T}_j, \quad \mathbb{T}_j = \mathbb{T} = \mathbb{R}/2\pi\mathbb{Z}, \quad \text{for all } j \in J.
\end{equation}
endowed with the product topology, that is, the topology of pointwise convergence. Note that any torus is metrisable and, by Tychonoff's theorem, it is compact. A basis for the product topology is given by the open cylinders, where a cylinder is a subset $\prod_{j \in J} X_j$ of $\mathbb{T}^J$ such that, for each $j \in J$, $X_j \subseteq \mathbb{T}$, with $X_j = \mathbb{T}$ for all but finitely many $j \in J$, and an open cylinder is a cylinder in which all the $X_j$ are open. See Appendix \ref{AAA} for more details about the product topology.

\begin{notation}\label{thetalift}
	Throughout this work, given $\theta=(\theta_j)_{j \in J} \in \T^J$, we denote by $\Theta:=(\Theta_j)_{j \in J}$ any list of real numbers such that $\theta_j=\Theta_j + 2\pi\Z \,$ for all $j \in J$.
\end{notation}

With this notation, we define the group operation $*$ on $\mathbb{T}^J$ componentwise: for any $\theta, \theta' \in \mathbb{T}^J$, their product $\theta * \theta'$ is the element of $\mathbb{T}^J$ whose $j$-th component is given by
\begin{equation}
(\theta * \theta')_j := \Theta_j + \Theta'_j + 2\pi\mathbb{Z}, \quad \text{for all } j \in J.
\end{equation}
It is straightforward to verify that this definition is well-defined, as it is independent of the choice of the representatives $\Theta_j$ and $\Theta'_j$.

\begin{rmk}
	In the following, we shall improperly use the term \textquotedblleft dimension\textquotedblright \ to refer to the rank (in the group-theoretic sense) of the torus. This is motivated by the fact that, if the cardinality of $J$ is finite, the rank of $\ \T^J$ coincides with its dimension when viewed as a topological manifold. Note that, if the cardinality of $J$ is not finite, $\T^J$ cannot be a topological manifold, i.e., locally homeomorphic to a topological vector space (TVS). Indeed, any open cylinder of $ \ \T^J$ contains non contractible loops while TVS are simply connected.
\end{rmk}

\paraga \textit{Kronecker Flow.} Given $\om=(\om_j)_{j \in J}\in \mathbb{R}^J$, we define the Kronecker flow associated with the \textquotedblleft frequency vector\textquotedblright $\, \om$ as the topological dynamical system $(\T^J,\Phi_\om)$ induced by the continuous $\mathbb{R}$-action
\begin{equation}\label{dyn}
\Phi_\om:\mathbb{R}\times \T^J \to \T^J, \quad (t,\theta) \mapsto \Phi^t_\om(\theta):=(\Theta_j +\om_j t +2\pi\Z)_{j \in J}.
\end{equation}
The real numbers $\om_j$ are called \textquotedblleft frequencies\textquotedblright.

\begin{rmk}\label{orb-Oom}
	The closure of any orbit induced by a Kronecker flow is homeomorphic to the closure of the orbit through the origin. This follows from the invariance by translations.
\end{rmk}

\subsection{Motivations}\label{motivations}
In the sixties, it was realized that many Hamiltonian systems with infinitely many degrees of freedom possess an infinite number of constants of motion. Some of these systems exhibit characteristics very similar to those of finite-dimensional completely integrable Hamiltonian systems. Namely, the solutions are supported on invariant tori of possibly infinite dimension; in this case, if an appropriate topology is induced on them, they turn out to be homeomorphic to $\T^J$ for some countable set of indices $J\subseteq \Z \,$, and the dynamics on them is conjugated to a Kronecker flow. Examples of systems having these features are provided by Hamiltonian PDEs on the circle having an elliptic fixed point, such as the Korteweg-de Vries equation (KdV), the Nonlinear Schr\"odinger equation (NLS) and the Benjamin-Ono equation (BO) \eqref{introBO} with periodic boundary conditions.

 These equations have been extensively studied in the Sobolev-Hilbert space $H^s_0(\T)$,
 \begin{equation}
 	H^s_0(\T):=\bigg\{u \in H^s(\T) \; : \; \int_{0}^{2\pi}u(x) \, dx=0\bigg\},
 \end{equation}
 for suitable $s\geq0$.\footnote{The flows of the equations KdV, NLS and BO preserve the average of the solutions.} For each of the above equations, there exists $s'\geq0$ and a global bi-analytic diffeomorphism
 \begin{equation*}
 \Psi:H_0^s(\T)\ni u \mapsto \Psi(u) \in h^{s'}:=\bigg\{z=(z_j)_{j \in \N}\in \C^\N \; : \; \|z\|^2_{h^{s'}}:=\sum_{j\in \N} j^{2s'}|z_j|^2<\infty\bigg\},
 \end{equation*}
 a \textquotedblleft Birkhoff map",\footnote{For the case of the NLS, being the unknown a complex valued function, $\N$ has to be replaced with $\Z\setminus\{0\}$.} such that , for any initial data $\underline{u}\in H^s_0(\T)$, the solution is supported on the invariant set $\Psi^{-1}(\mathcal{T}_{I})$. Here $I=(I_j)_{j \in \N}$ is the list defined by $I_j:=|\Psi_j(\underline{u})|^2$, where $\Psi(\underline{u})=(\Psi_j(\underline{u}))_{j \in \N}$, and
 \begin{equation}\label{Tgamma}
 \mathcal{T}_{I}:=\{z \in h^{s'} \; : \; |z_j|^2=I_j, \; \forall j\in \N  \}.
 \end{equation}
  \begin{rmk}
 	The sets defined in \eqref{Tgamma}, as showed in Appendix \ref{AAC}, when considered with the subspace topology induced by the $\|\cdot\|_{h^{s'}}$-norm, are homeomorphic to the possibly infinite torus $\T^{\text{supp}(I)}$, where $\text{supp}(I)$ denotes the support of the sequence $I$. Namely, the map
 	\begin{equation}
 	\Xi:\T^{\text{supp}(I)} \ni \theta \mapsto \Xi(\theta):=(\sqrt{I_j}e^{\im \Theta_j})_{j\in \N}\in \mathcal{T}_{I}\subset h^{s'},
 	\end{equation}
 	where one chooses $\Theta_j:=0$ for all $j\notin \text{supp}(I)$, is a topological embedding. Whenever the cardinality of $\text{supp}(I)$ is infinite, the torus $\T^{\text{supp}(I)}$ has infinite dimension. 
 	\end{rmk}
\noindent With respect to the coordinate system induced by the Birkhoff map, the solution for the given initial data $\underline{u}$ is continuous in time as a map with values in $h^{s'}$ and is given by
 	\begin{equation}\label{f}
 	\Psi_j(u(t))=\Psi_j(\underline{u})e^{\im \om_j(I)t}, \quad 	\om_j(I):=\frac{\partial H\circ \Psi^{-1}}{\partial |z_j|^2}(\sqrt{I_1},...,\sqrt{I_n},...), \quad \forall j \in \N,
 	\end{equation}
 	where $H:H^s_0(\T)\to \R$ is the Hamiltonian function of the system under consideration.
For references, see \cite{kappeler2003kdv},\cite{BO},\cite{GrebKappaler2014} and \cite{Kappeler2009}. Therefore, since the Birkhoff map $\Psi$ is a bi-analytic diffeomorphism, and hence, a homeomorphism, the Hamiltonian flow on the invariant torus $\Psi^{-1}(\mathcal{T}_I)$ is topologically conjugate by $\mathscr{H}:=\Psi^{-1}\circ \Xi \ $ to the Kronecker flow
\begin{equation}
 	\Phi_\om:\mathbb{R}\times \T^{\text{supp}(I)} \to \T^{\text{supp}(I)}, \quad (t,\theta) \mapsto \Phi^t_\om(\theta):=(\Theta_j +\om_j(I) t +2\pi\Z)_{j\in \text{supp}(I)},
 	\end{equation}
 	with $\om=(\om_j(I))_{j \in \N}$.

In \cite{KM}, the global weak* $\mathscr{C}^0$-wellposedness of the periodic KdV equation in the space of pseudomeasures on $\T$ is proved. For any given initial data, the solution is supported on an invariant subset that, when endowed with the subspace topology induced by the weak$^*$ topology, is homeomorphic to a torus  (possibly of infinite dimension) with the product topology. Furthermore, the KdV flow on these invariant sets is topologically conjugate to a Kronecker flow.

To underline the ubiquity of Kronecker flows in the context of infinite Hamiltonian systems, we point out that there are also results about the construction of invariant tori for non-integrable Hamiltonian PDEs. For instance, in \cite{BOURGAIN200562}, Bourgain constructs invariant tori of \textquotedblleft full dimension" in Gevrey regularity for the $1D$ periodic NLS equation with convolution potential $V=(V_j)_{j \in \Z}\in \ell^\infty(\Z,\R)$ whose coefficients are considered as external parameters. In \cite{BMP}, the authors extend the result of Bourgain to an entire ball of initial data in the phase space of Fourier coefficients
\begin{equation}\label{w}
	\tw^\infty_{p,s,a}:=\bigg\{u=(u_j)_{j \in \Z}\in\C^\Z \; : \; \|u\|_{\tw^\infty_{p,s,a}}:=\sup_{j \in \Z}|u_j|\jap{j}^pe^{a|j|+s\jap{j}^\alpha} \bigg\}, \quad \jap{j}:=\max\{1,|j|\},
\end{equation}
with $p>1$, $s>0$, $0<\alpha<1$ and suitable $a\geq 0$. Namely, for most choices of \mbox{$\omega = (\omega_j)_{j \in \mathbb{N}} \in \mathbb{R}^\mathbb{Z}$} satisfying $\sup_{j \in \Z}|\om_j-j^2|\leq1/2$, and for all list $I=(I_j)_{j \in \Z}$ of non-negative real numbers such that $(\sqrt{I_j})_{j \in \Z}$ belongs to a suitable neighbourhood of the origin in $\tw^\infty_{p,s,a}$, there exists $V\in \ell^\infty(\Z,\R)$ such that the set $\mathcal{T}_I$ defined by
\begin{equation}
\mathcal{T}_{I}:=\{u=(u_j)_{j \in \Z}\in \tw^\infty_{p,s,a} \; : \; |u_j|^2=I_j \}
\end{equation}
is invariant for the NLS flow with convolution potential $V$. Furthermore, when $\mathcal{T}_I$ is endowed with the subspace topology induced by the weak* topology $\tau_*$,\footnote{Being the topological dual of the Banach space
	$$\ell^1_{p,s,a}:=\{v=(v_j)_{j \in \Z} \; : \; \|v\|_{\ell^1_{p,s,a}}:=\textstyle{\sum_{j \in \Z}}|v_j|\jap{j}^{-p}e^{-a|j|-s\jap{j}^\alpha}<\infty\},$$
	$\tw^\infty_{p,s,a}$ can be endowed with the weak* topology, i.e., the coarsest topology that makes $\tw^\infty_{p,s,a}\ni u \mapsto u(v) \in \C$ continuous for all $v \in \ell^1_{p,s,a}$.} the dynamics on such invariant set is topologically conjugate to a Kronecker flow on $\T^{\text{supp}(I)}$ (with $\om_j \sim j^2$ for large $|j|$) through the homeomorphism associated with the embedding\footnote{See Appendix \ref{AAC}.}
\begin{equation}
\mathbb{T}^{\text{supp}(I)} \ni \theta \mapsto (\sqrt{I_j} e^{\im \Theta_j})_{j \in \mathbb{Z}} \in \mathcal{T}_I \subset (\tw^\infty_{p,s,a}, \tau_*), \quad \Theta_j := 0, \; \text{for all } j \notin \text{supp}(I).
\end{equation}

The analysis carried out in the present paper yields a connection between algebraic properties of the characteristic frequencies of the infinite Hamiltonian system under consideration and the topological structure of the minimal invariant sets. In Section \ref{ssbo} we apply our results to the Benjamin-Ono equation \eqref{introBO}.

Before continuing with the introduction, it is worth noting that Kronecker flows on infinite tori are deeply connected with the theory of almost-periodic functions. We refer the interested reader to \cite{fink1974almost}.

\subsection{Presentation of the Main Results}\label{pres}
In order to present our results, we briefly recall some basic definitions from topological dynamics.

\paraga \textit{Basics of Topological Dynamics.}  A topological dynamical system is a pair $(X,\Phi)$, where $X$ is a topological space and $\Phi$ is a continuous flow $\Phi:\mathbb{R}\times X \to X$. An invariant probability measure for $(X,\Phi)$ is a probability Borel measure $\mu$ on $X$ such that, for any $t$ in $\mathbb{R}$ and any Borel subset $U\subseteq X$, $\mu((\Phi^{t})^{-1}(U))=\mu(U)$. In terms of the pushforward measure, this states that $\Phi^t_*(\mu)=\mu$ for all $ t \in \mathbb{R}$. 

We now recall the definitions of the properties of a topological dynamical system that will be relevant to our study.

A topological dynamical system $(X,\Phi)$ is
\begin{itemize}
	\item  topologically transitive if there exists a dense orbit, and minimal if every orbit is dense;
	\item  uniquely ergodic if it admits a unique invariant probability measure.
\end{itemize}

Let $\mu$ be an invariant probability measure for $(X,\Phi)$. Then, $(X,\Phi,\mu)$ is
\begin{itemize}
\item ergodic if the only invariant sets have measure $1$ or $0$.
\end{itemize}

\begin{rmk}\label{uni}
	A topological dynamical system $\Phi:\mathbb{R}\times X \to X$ admitting a unique invariant probability measure must be ergodic. Indeed, let $\mu$ be an invariant probability measure for $(X,\Phi)$. If the system is not ergodic, there exists $U \subset X$ invariant and such that $0<\mu(U)<\mu(X)=1$. But then, we can define another invariant probability measure simply normalizing by $\mu(U)$, namely, $\mu_U(V):=(\mu(U))^{-1}\mu(V\cap U), \ \forall V \subset X$.
\end{rmk}

A central notion in topological dynamics is that of topological conjugacy. In particular, all the aforementioned properties of a topological dynamical system are preserved under topological conjugacy.

\begin{defn}\label{topoconj}
	Two topological dynamical systems $(X,\Phi)$ and $(Y,\Psi)$ are said to be topologically conjugate if there exists a homeomorphism $h:X \to Y $ such that, for any $t \in \mathbb{R}$, 
	\begin{equation}
		h\circ \Phi^t =\Psi^t \circ h.
	\end{equation}

\end{defn}

\paraga \textit{Resonance Module and Frequency Module of a Kronecker flow.}
In this paragraph, we introduce two abelian groups that one can associate with every Kronecker flow: the resonance module and the frequency module, both of which play a crucial role in the following discussion.

\begin{rmk}
	Given a countable set of indices $J$, and a countable family of abelian groups $(G_j)_{j \in J}$, the direct sum $\bigoplus_{j \in J}G_j$ is the abelian subgroup of $\prod_{j \in J}G_j$ consisting of all $(g_j)_{j \in J}\in\prod_{j \in J}G_j$ such that $g_j$ differs from $0$ only for finitely many $j \in J$.
\end{rmk}
\begin{notation}
	Given a set of indices $J$ and an abelian group $G$, we adopt the convention
	\begin{gather}
	G^J:=\prod_{j\in J}G:=\prod_{j \in J}G_j, \quad G_j=G, \; \forall j \in J, \\ G^{(J)}:=\bigoplus_{j\in J}G:=\bigoplus_{j \in J}G_j, \quad G_j=G, \; \forall j \in J.
	\end{gather}
	If $J=\{1,...,n\}$, with $n\in \N$, we set $G^n:=G^J=G^{(J)}$.
\end{notation}
\begin{defn}[Resonance module]
	Given $\om \in \mathbb{R}^J$, we define the resonance module associated with $\om$ as 
	\begin{equation}\label{resmod}
	\mathscr{R}_\om:=\bigg\{\nu \in \Z^{(J)} \; : \; \sum_{j \in J}\om_j\nu_j=0\bigg\}.
	\end{equation}
	We say that, the real numbers $(\om_j)_{j \in J}$ are rationally independent if and only if $\mathscr{R}_\om=\{0\}$, in which case, we say that $\om=(\om_j)_{j \in J}$ is non-resonant, otherwise we say that $\om$ is resonant. Accordingly, the Kronecker flow $(\T^J,\Phi_\om)$ associated with $\om$ is said to be (non-)resonant if $\om$ is (non-)resonant.
\end{defn}

\begin{defn}\label{fa}
	An abelian group is said to be free if it admits a basis. Namely, an abelian group $G$ is free if and only if there exists a set of indices $J$ and an isomorphism of groups between $G$ and $\Z^{(J)}$ (see Corollary \ref{freeZ}).
\end{defn}

\begin{rmk}
	Being a subgroup of the free abelian group $\Z^{(J)}$, the resonance module $\mathscr{R}_\om$ is itself free as an abelian group (see, for e.g., \cite{rotman2002advanced}).
\end{rmk}

\begin{defn}[Frequency Module]
	The frequency module associated with $\om\in\mathbb{R}^J$ is the subgroup of $(\R,+)$ generated by the integer (finite) linear combinations of the frequencies, i.e.,
	\begin{equation}
	\mathscr{M}_\om:=\bigg\{\sum_{j\in J}\om_j\nu_j \; : \; \nu \in \Z^{(J)}\bigg\}\subset \R .
	\end{equation}
\end{defn}

\begin{rmk}\label{firstisothm}
	As the map $\Z^{(J)} \ni \nu \mapsto \om\cdot\nu \in (\R,+)$ is a group homomorphism, by the first isomorphism theorem, the frequency module $\mathscr{M}_\om$ is isomorphic to the quotient group $\Z^{(J)}/\mathscr{R}_\om$, where $\mathscr{R}_\om$ is the resonance module defined in \eqref{resmod}.
\end{rmk}

\begin{rmk}\label{cwt}
$\mathscr{M}_\om$ is countable and without torsion, namely, for each element $g\in\mathscr{M}_\om$, if there exists a non-zero integer $n$ such that $ng=0$, then, $g=0$.
\end{rmk}
 The following remark highlights the most fundamental difference between the finite-dimensional and infinite-dimensional settings and, as we shall see in Theorem \ref{B}, reveals the reasoning behind why the simple, well known properties \ref{prop:nonres} and \ref{prop:kronecker} that hold for Kronecker flows on finite-dimensional tori do not apply in general to this infinite-dimensional context.
\begin{rmk}\label{remfinfree}
If the cardinality of $J$ is finite, then $\mathscr{M}_\om$ is a free abelian group: this follows from Remark \ref{cwt} and the general structure theorem for finitely generated abelian groups without torsion, noting that the finiteness of the cardinality of $J$ ensures that the frequency module is finitely generated.

\smallskip
If the cardinality of $J$ is not finite, $\mathscr{M}_\om$ may be not free: it is no longer guaranteed that $\mathscr{M}_\om$ is finitely generated. As an example, for $J=\N$ and $\om=(\om_j)_{j \in \N}$ with $\om_j=1/j$, one has $\mathscr{M}_\om=(\Q,+)$.
\end{rmk}

\paraga \textit{The Finite-Dimensional Case.} Let $n \in \mathbb{N}$ and $\omega = (\omega_j)_{j=1}^n \in \mathbb{R}^n$. For the Kronecker flow $\Phi_\omega$ on $\mathbb{T}^n$, topological transitivity, minimality, and unique ergodicity are equivalent to the non-resonance condition $\mathscr{R}_\omega = \{0\}$.

Furthermore, if $1 \leq \mathrm{rank}(\mathscr{R}_\omega) =: d < n$, there exists an automorphism $\mathscr{A}$ of $\mathbb{T}^n$ that conjugates $(\mathbb{T}^n, \Phi_\omega)$ to another Kronecker flow $(\mathbb{T}^n, \Phi_{\tilde{\omega}})$; specifically,
\begin{equation}
\mathscr{A} \circ \Phi^t_{\omega} = \Phi^t_{\tilde{\omega}} \circ \mathscr{A} \quad \text{for all } t \in \mathbb{R},
\end{equation}
where $\tilde{\omega} = (0_d, \bar{\omega})$, $\bar{\omega}$ is non-resonant, and $0_d$ is the null vector in $\mathbb{R}^d$. In particular, if we denote by $q^d$ the projection
\begin{equation}
q^d : \mathbb{T}^n \to \mathbb{T}^d, \quad (\theta_j)_{j=1}^n \mapsto (\theta_j)_{j=1}^d,
\end{equation}
the continuous surjection $q^d \circ \mathscr{A} : \mathbb{T}^n \to \mathbb{T}^d$ defines a (trivial) fiber bundle with fibers homeomorphic to $\mathbb{T}^{n-d}$. Each fiber is a minimal invariant set for the Kronecker flow $(\mathbb{T}^n, \Phi_\omega)$, and the restriction of the flow to each fiber is topologically conjugate to the non-resonant Kronecker flow $(\mathbb{T}^{n-d}, \Phi_{\bar{\omega}})$.

\paraga \textit{The Main Results.} Our main concern in this paper is the study of the case $J=\N$. We refer to $\T^\N$ as the infinite torus. 
\medskip

In Section \ref{ue}, we show that for non-resonant Kronecker flows on $\T^\N$ hold the same properties as in the finite-dimensional case. More precisely, the first purpose is to give a proof of the following assertion.
\begin{thmA}
The following properties of a Kronecker flow on $\T^\N$ are equivalent:
\begin{itemize}
	\item[(a)] it is non-resonant;
	\item[(b)] it is uniquely ergodic;
	\item[(c)] it is minimal;
	\item[(d)] it is topologically transitive.
\end{itemize}
\end{thmA}
This theorem should be compared with \cite{Jess}, where the author essentially proves that the average of any integrable function along the orbits of a non-resonant Kronecker flow on the infinite torus equals the average over the entire torus with respect to an appropriate invariant integral. The ergodic properties of Kronecker flows on infinite-dimensional tori are also studied in \cite{kozov} and \cite{SakVol}.

It is worth stressing that non-resonant Kronecker flows are in a certain sense \textquotedblleft typical". This was first observed by Bourgain \cite{BOURGAIN200562} who proved that, for all $\delta>0$, the set
\begin{equation}
	\Omega_\delta:=\bigg\{\om \in [-1,1]^\N \; : \; \sum_{j \in \N}\om_j\nu_j \geq \delta \prod_{j \in \N}(1+\nu_j^2j^4)^{-1}, \ \forall \nu \in \Z^{(\N)} \bigg\}
\end{equation}
has measure $\mathfrak{m}(\Omega_\delta)=1-O(\delta)$, where $\mathfrak{m}$ is the product Lebesgue measure on $[-1,1]^\N$.
\medskip

Next we consider the resonant case. In the infinite-dimensional setting, Property \ref{prop:kronecker} does not hold in general: in the above mentioned paper \cite{SakVol}, Sakbaev and Volovich give examples of linear flows on the infinite-dimensional torus where the closures of the orbits are not tori, even though these orbits are accumulated by periodic ones. However, a legitimate question to ask is whether a Kronecker flow on $\T^\N$ is topologically conjugate to another Kronecker flow whose non-vanishing frequencies are rationally independent. In Section \ref{secprel} we present an algorithm to reduce resonant Kronecker flows on the infinite torus in the presence of finitely many independent resonances. More precisely, in Proposition~\ref{propprel} we show that for a Kronecker flow $(\mathbb{T}^\mathbb{N}, \Phi_\omega)$ where the resonance module $\mathscr{R}_\omega$ has finite dimension $1\leq d <\infty$, the system is topologically conjugate to another Kronecker flow $(\mathbb{T}^\mathbb{N}, \Phi_{\tilde{\omega}})$ characterized by a frequency vector $\tilde{\omega}=(0_d, \bar{\omega})$ where $0_d$ is the null vector in $\mathbb{R}^d$ and $\bar{\omega}$ is non-resonant. Aiming to fill the gap left by the assumption of a finite-dimensional resonance module, we establish the following more general result, which connects an algebraic property of the frequency vector with the $\mathscr{C}^0$-conjugacy class to which the resonant Kronecker flow belongs (see Section \ref{thefreecase} for the proof).
	
\begin{thmB}
Let $\om \in \R^\N$ and let $(\T^\N,\Phi_\om)$ be the corresponding Kronecker flow. Then, $\mathscr{M}_{\om}$ is a free abelian group if and only if $(\T^\N,\Phi_\om)$ is topologically conjugate to another Kronecker flow $(\T^\N,\Phi_{\tilde{\om}})$ whose non-vanishing frequencies are rationally independent. If this is the case, $\mathscr{M}_{\om}$ is isomorphic to $\Z^{(\text{supp}(\tilde{\om}))}$.
\end{thmB}
In particular, if the frequency module is free, $\T^\N$ is fibered by invariant, minimal embedded tori such that the projection of the flow on each fiber is conjugate to a non-resonant Kronecker flow. This implies that the closure of any orbit is still a torus, of possibly infinite dimension. 

Note that Property \ref{prop:kronecker} of the finite-dimensional case follows from the fact that, as already pointed out in Remark \ref{remfinfree}, if the frequency vector has finite support, then $\mathscr{M}_\om$ is finitely generated, and hence, by the structure theorem for finitely generated abelian groups without torsion, it is free.

Furthermore, in the presence of finitely many resonances, Proposition~\ref{propprel} provides a conjugacy to a flow $(\mathbb{T}^\mathbb{N}, \Phi_{\tilde{\omega}})$ whose non-vanishing frequencies are rationally independent. By Theorem~\ref{B}, this ensures that $\mathscr{M}_\omega$ must be a free abelian group, even though it is necessarily no longer finitely generated.

\medskip

Nevertheless, in the general case, linear dynamics on an infinite torus can be much more complex. In Theorem \ref{thmsol} of Section \ref{ratflow}, we construct a special class of Kronecker flows with the property that the closure of any orbit is a solenoid (see Definition \ref{Solenoid}), that is a compact space locally homeomorphic to the product of an interval and a Cantor set. By Lemma \ref{modsol}, the frequency module of any Kronecker flow in such class is a non-free subgroup of $(\Q,+)$.

Next, we show that the closure of any orbit induced by a Kronecker flow on the infinite torus is homeomorphic to the Pontryagin dual of the associated frequency module (Proposition \ref{mainprel}). At the end of Section \ref{orbit-dual}, we use this result, together with Theorem \ref{thmsol} and classical results from Pontryagin's theory of locally compact abelian groups, to prove the following statement.
\begin{thmC}
	Let $I$ be a countable set of indices and $(G_i)_{i \in I}$ a family of subgroups of $(\mathbb{Q},+)$. Then, there exists a Kronecker flow whose frequency module is isomorphic to $\bigoplus_{i \in I} G_i$. Moreover, the closure of any orbit is homeomorphic to a product $\prod_{i \in I} X_i$, where each factor $X_i$ is a circle if $G_i$ is free, and a solenoid otherwise.
\end{thmC}
This is an extension of the recent result in \cite{SakVol} concerning the existence of Kronecker flows on the infinite torus inducing orbits whose closure is not a torus (see Remark \ref{rmksakvol}), as we exhibit a well-defined class and we provide a more refined understanding of the topological structure.
\medskip

In Section \ref{ssbo} we apply these results to the Benjamin-Ono equation, namely, we prove the following
\begin{thmBO}
Consider the Benjamin-Ono equation for a real-valued function $u(t,x)$ that is $2\pi$-periodic in space:
\begin{equation}
u_t=\mathrm{H}(u_{xx})-2uu_x, \quad x \in \mathbb{T}, \; t \in \mathbb{R},
\end{equation}
where $\mathrm{H}(\cdot)$ denotes the Hilbert transform defined in \eqref{ht0}. Let
\begin{equation}
\Psi:L_0^2(\mathbb{T})\ni u \mapsto z \in h^{1/2}
\end{equation}
be its Birkhoff map \cite{BO}. Given $z\in h^{1/2}$, if there exists $\beta\in \mathbb{R}\setminus \mathbb{Q}$ such that $|z_j|^2\in \beta \mathbb{Q}\setminus\{0\}$ for all $j \in \mathbb{N}$, then the closure of the orbit with initial data $u = \Psi^{-1}(z)$ under the Benjamin-Ono flow is homeomorphic to the product of a circle and a solenoid.
\end{thmBO}
It would be of interest to determine whether this family of almost-periodic solutions corresponds to specific physical features, such as a novel form of turbulence.
\medskip

We conclude this work by discussing a general classification problem for Kronecker flows on the infinite torus. First, we prove the following statement.

\begin{thmD}
	Given $\om,\tilde{\om} \in \R^\N$, let $(\mathbb{T}^\mathbb{N}, \Phi_\omega)$ and $(\mathbb{T}^\mathbb{N}, \Phi_{\tilde{\omega}})$ be two Kronecker flows on the infinite torus. The closure of any orbit of $\Phi_\omega$ is homeomorphic to the closure of any orbit of $\Phi_{\tilde{\omega}}$ if and only if the associated frequency modules $\mathscr{M}_\omega$ and $\mathscr{M}_{\tilde{\omega}}$ are isomorphic as abelian groups.
\end{thmD}

Finally, we show how to associate to each countable abelian group $G$ without torsion a Kronecker flow $(\T^\N,\Phi_\om)$ in such a way that $G$ is isomorphic to $\mathscr{M}_\om$. Therefore, by Remark \ref{cwt}, Theorem \ref{E} implies that completely classifying Kronecker flows on $\T^\N$ up to homeomorphism of the closure their orbits is a highly complex issue, as it is equivalent to the classification problem for countable abelian groups without torsion. To date, a complete classification is known only for those groups having rank\footnote{See Appendix \ref{AAB} for the definition of rank and other classical results from the theory of abelian groups.} $1$, namely, additive subgroups of the rationals (see \cite{baer}, \cite{beau} and \cite{Fuchs2015}) which corresponds to the class of Kronecker flows (constructed in Section \ref{orbit-dual}) whose orbits are dense in a product of circles and solenoids, in the sense of Theorem \ref{productcirclesol}.

\section{Preliminaries}\label{secprel}
In this section, we provide a characterization of the group $\mathrm{Aut}(\mathbb{T}^\mathbb{N})$ of automorphisms of the topological group $\mathbb{T}^\mathbb{N}$ (see Proposition \ref{matrix}) and apply this result to develop an algorithm for reducing resonant Kronecker flows (see Proposition \ref{propprel}).

\paraga \textit{Characterization of $\mathrm{Aut}(\T^\N)$.}
\begin{notation}\label{notmap}
	We shall adopt the following notations throughout this section:
	\begin{itemize}
		\item For any $j \in \mathbb{N}$, $q_j: \mathbb{T}^\mathbb{N} \to \mathbb{T}$ denotes the projection onto the $j$-th factor, namely, $q_j(\theta) = \theta_j$.
		\item For any $m \in \mathbb{N}$, $q^m: \mathbb{T}^\mathbb{N} \to \mathbb{T}^m$ denotes the projection $(\theta_j)_{j \in \mathbb{N}} \mapsto (\theta_j)_{j=1}^m$.
		\item For any $m \in \mathbb{N}$, $i^m: \mathbb{R}^m \to \mathbb{R}^\mathbb{N}$ denotes the inclusion $(\Theta_j)_{j=1}^m \mapsto (\Theta_1, \dots, \Theta_m, 0, \dots)$.
		\item For any $m \in \mathbb{N}$, $\pi_m: \mathbb{R}^m \to \mathbb{T}^m$ denotes the standard quotient projection $\Theta \mapsto \Theta + (2\pi\mathbb{Z})^m$.
		\item $\pi_\mathbb{N}: \mathbb{R}^\mathbb{N} \to \mathbb{R}^\mathbb{N}/(2\pi\mathbb{Z})^\mathbb{N}$ denotes the projection $\Theta \mapsto \Theta + (2\pi\mathbb{Z})^\mathbb{N}$.
	\end{itemize}
\end{notation}

\begin{rmk}
	All the spaces defined above are endowed with the product topology. This choice ensures that all the maps introduced in Notation \ref{notmap} are continuous.
\end{rmk}

We first show that $\mathbb{T}^\mathbb{N}$ is naturally isomorphic as a topological group to the quotient $\mathbb{R}^\mathbb{N}/(2\pi\mathbb{Z})^\mathbb{N}$, where $\mathbb{R}^\mathbb{N}$ is endowed with the product topology and the quotient carries the induced quotient topology.

\begin{prop}[Infinite torus as quotient space]\label{quoz}
	The torus $\mathbb{T}^{\mathbb{N}}$ is isomorphic as a topological group to the quotient $\mathbb{R}^{\mathbb{N}} / (2\pi\mathbb{Z})^{\mathbb{N}}$.
\end{prop}

\begin{proof}
	The algebraic group isomorphism is immediate. To prove that the map
	$$ \Phi: \mathbb{R}^{\mathbb{N}}/(2\pi\mathbb{Z})^{\mathbb{N}} \to \mathbb{T}^{\mathbb{N}}, \quad \Theta + (2\pi\mathbb{Z})^{\mathbb{N}} \mapsto (\Theta_j + 2\pi\mathbb{Z})_{j \in \mathbb{N}} $$
	is a homeomorphism, we recall that a set $U \subseteq \mathbb{R}^{\mathbb{N}}/(2\pi\mathbb{Z})^{\mathbb{N}}$ is open if and only if its preimage $\pi_\mathbb{N}^{-1}(U)$ is open in $\mathbb{R}^{\mathbb{N}}$.
	
Since the product topology on $\mathbb{T}^\mathbb{N}$ is generated by the subbase of open cylinders of the form $q_j^{-1}(A)$ (with $A \subseteq \mathbb{T}$ open), it suffices to check that their preimages under $\Phi$ are open. To this end, observe that
$$ \pi_\mathbb{N}^{-1}\left(\Phi^{-1}(q_j^{-1}(A))\right) = \{ \Theta \in \mathbb{R}^\mathbb{N} \; : \; \Theta_j \in \pi^{-1}(A) \}, $$
where $\pi: \mathbb{R} \to \mathbb{T}$ is the standard projection. This set is an open cylinder in $\mathbb{R}^\mathbb{N}$; hence, $\Phi$ is continuous.
	
Conversely, the image under $\Phi \circ \pi_{\mathbb{N}}$ of any open cylinder in $\mathbb{R}^{\mathbb{N}}$ is an open cylinder in $\mathbb{T}^{\mathbb{N}}$. Since $\Phi \circ \pi_{\mathbb{N}}$ maps the basis of the product topology of $\mathbb{R}^{\mathbb{N}}$ to the basis of $\mathbb{T}^{\mathbb{N}}$, it follows that $\Phi \circ \pi_{\mathbb{N}}$ is an open map. By the properties of the quotient topology, this implies that $\Phi$ is an open map as well, and thus a homeomorphism.
\end{proof}

\begin{rmk}[Not a covering map]\label{notcov}
	To avoid misunderstandings, we stress that the previous proposition does not imply that $\pi_{\mathbb{N}}$ is a covering map. In fact, $\pi_{\mathbb{N}}$ cannot be a local homeomorphism, since $\mathbb{R}^{\mathbb{N}}$ is simply connected, whereas inside each cylinder of the torus one can find non-contractible loops.
\end{rmk}

\begin{rmk}[Abuse of notation]\label{abuse}
	In what follows, we shall identify $\mathbb{T}^{\mathbb{N}}$ with $\mathbb{R}^{\mathbb{N}} / (2\pi\mathbb{Z})^{\mathbb{N}}$ and $\mathbb{T}^m$ with $\mathbb{R}^m / (2\pi\mathbb{Z})^m$ in the sense of Proposition \ref{quoz}.
\end{rmk}

\begin{thm}[Lifts]\label{lifts}
	Given a continuous function $F : \mathbb{T}^{\mathbb{N}} \to \mathbb{T}^{\mathbb{N}}$ such that $F(0) = 0$, there exists a unique continuous function $\hat{F} : \mathbb{R}^{\mathbb{N}} \to \mathbb{R}^{\mathbb{N}}$, which we call \textquotedblleft lift", such that $\pi_{\mathbb{N}} \circ \hat{F} = F \circ \pi_{\mathbb{N}}$ and $\hat{F}(0) = 0$.
\end{thm}

\begin{proof}
	For every $m \in \mathbb{N}$, consider the map
	$$ f^m := q^m \circ F \circ \pi_{\mathbb{N}} : \mathbb{R}^{\mathbb{N}} \to \mathbb{T}^m. $$
	Since $\mathbb{R}^{\mathbb{N}}$ is connected and simply connected, and $\pi_m : \mathbb{R}^m \to \mathbb{T}^m$ is the universal covering of $\mathbb{T}^m$ with $\pi_m(0) = 0$, there exists a unique continuous lift $\hat{f}^m : \mathbb{R}^{\mathbb{N}} \to \mathbb{R}^m$ such that $\pi_m \circ \hat{f}^m = f^m$ and $\hat{f}^m(0) = 0$.
	
	We then define $\hat{F}^m : \mathbb{R}^{\mathbb{N}} \to \mathbb{R}^{\mathbb{N}}$ via the inclusion $i^m$:
	$$ \hat{F}^m := i^m \circ \hat{f}^m. $$
	By construction, each $\hat{F}^m$ is continuous and satisfies $\hat{F}^m(0) = 0$. We now show that for any $j \leq m$, the $j$-th component $\hat{F}^m_j$ is independent of $m$.
	
	Observe that for $j \leq m$, the component $f_j^m = (F \circ \pi_{\mathbb{N}})_j$ does not depend on $m$. Thus, both $\hat{F}_j^m$ and $\hat{F}_j^{m+1}$ are lifts of the same map $f_j^m : \mathbb{R}^{\mathbb{N}} \to \mathbb{T}$. Specifically:
	$$ \hat{F}_j^m(\Theta) + 2\pi\mathbb{Z} = f_j^m(\Theta) = f_j^{m+1}(\Theta) = \hat{F}_j^{m+1}(\Theta) + 2\pi\mathbb{Z}. $$
	Since $\mathbb{R}^\mathbb{N}$ is connected and $\hat{F}_j^m(0) = \hat{F}_j^{m+1}(0) = 0$, the uniqueness of lifts implies that:
	$$ \hat{F}_j^m = \hat{F}_j^{m+1} \quad \text{for all } m \geq j. $$
	
	This consistency allows us to well-define the map $\hat{F} : \mathbb{R}^{\mathbb{N}} \to \mathbb{R}^{\mathbb{N}}$ by setting, for each $j \in \mathbb{N}$:
	$$ \hat{F}_j := \hat{F}_j^m \quad \text{for any } m \geq j. $$
	By definition, $\hat{F}(0) = 0$ and $\pi_{\mathbb{N}} \circ \hat{F} = F \circ \pi_{\mathbb{N}}$ holds component-wise. Finally, $\hat{F}$ is continuous because each of its components $\hat{F}_j$ is continuous, which is the defining property of the product topology.
	
	Regarding uniqueness, any other lift $\tilde{F}$ satisfying the same conditions would necessarily have components $\tilde{F}_j$ which are continuous lifts of $f_j = (F \circ \pi_{\mathbb{N}})_j$ such that $\tilde{F}_j(0)=0$. Since $\pi: \mathbb{R} \to \mathbb{T}$ is a covering map and $\mathbb{R}^\mathbb{N}$ is connected, each component $\tilde{F}_j$ is uniquely determined, whence $\tilde{F} = \hat{F}$.
\end{proof}

Let $\mathrm{Aut}(\mathbb{T}^\mathbb{N})$ be the group of automorphisms of $\mathbb{T}^\mathbb{N}$, namely, the set of bijective maps $\mathscr{A} : \mathbb{T}^\mathbb{N} \to \mathbb{T}^\mathbb{N}$ such that both $\mathscr{A}$ and $\mathscr{A}^{-1}$ are continuous group homomorphisms. To characterize these automorphisms, we first introduce a particular group of infinite invertible matrices.

\begin{defn}[The group $\mathrm{FGL}_{\mathbb{N}}(\mathbb{Z})$]
	We denote by $\mathrm{FGL}_{\mathbb{N}}(\mathbb{Z})$ the set of infinite matrices $A = (a_{ij})_{i,j \in \mathbb{N}}$ with $a_{ij} \in \mathbb{Z}$ satisfying the following conditions:
	\begin{enumerate}
		\item each row of $A$ has finite support, i.e., for every $i \in \mathbb{N}$, the set $\{j \in \mathbb{N} : a_{ij} \neq 0\}$ is finite;
		
		\item the matrix $A$ is invertible, meaning there exists an infinite matrix $A^{-1} = (a'_{ij})_{i,j \in \mathbb{N}}$ with integer entries and finite support rows such that $A A^{-1} = A^{-1} A = I$, where $I$ is the infinite identity matrix.
	\end{enumerate}
\end{defn}

The condition on the finite support of the rows ensures that the action of the matrix on $\mathbb{R}^\mathbb{N}$ is well-defined. Specifically, for any $\Theta = (\Theta_j)_{j \in \mathbb{N}} \in \mathbb{R}^\mathbb{N}$, the $i$-th component of the image $A\Theta$ is given by the sum:
\begin{equation}
(A\Theta)_i = \sum_{j \in \mathbb{N}} a_{ij}\Theta_j.
\end{equation}
Since each row has finite support, this sum is always finite, and thus $A\Theta$ is a well-defined element of $\mathbb{R}^\mathbb{N}$. Furthermore, the resulting linear map $A: \mathbb{R}^\mathbb{N} \to \mathbb{R}^\mathbb{N}$ is continuous with respect to the product topology, as each component $(A\Theta)_i$ depends only on a finite number of coordinates of $\Theta$.

Matrix multiplication for $A, B \in \mathrm{FGL}_{\mathbb{N}}(\mathbb{Z})$ is defined by the standard rule $(AB)_{ij} = \sum_{k \in \mathbb{N}} a_{ik}b_{kj}$. This sum is always well-posed because each row of $A$ contains only finitely many non-zero entries. Moreover, $\mathrm{FGL}_{\mathbb{N}}(\mathbb{Z})$ is closed under multiplication: the $i$-th row of the product $AB$ is a finite linear combination of the rows of $B$, namely 
$$(AB)_{i \cdot} = \sum_{k \in S_i} a_{ik} (B)_{k \cdot},$$ 
where $S_i := \{k \in \mathbb{N} : a_{ik} \neq 0\}$ is the finite support of the $i$-th row of $A$. Since the union of finitely many finite sets is itself finite, the $i$-th row of $AB$ also has finite support. This structure, combined with the existence of the inverse (which by definition belongs to the same set), ensures that $\mathrm{FGL}_{\mathbb{N}}(\mathbb{Z})$ forms a group under matrix multiplication.

\begin{prop}\label{matrix}
	$\mathrm{Aut}(\mathbb{T}^\mathbb{N})$ is isomorphic to $\mathrm{FGL}_{\mathbb{N}}(\mathbb{Z})$ via the correspondence
	\begin{equation}\label{isomatrix}
	\mathrm{FGL}_\mathbb{N}(\mathbb{Z}) \ni A=(a_{ij})_{i,j \in \N} \mapsto \mathscr{A} \in \mathrm{Aut}(\mathbb{T}^\mathbb{N}),
	\end{equation}
	where, for all $\theta = (\Theta_j + 2\pi\mathbb{Z})_{j \in \mathbb{N}} \in \mathbb{T}^\mathbb{N}$, the map $\mathscr{A}$ is defined by
	\begin{equation}\label{isomatrix2}
	\mathscr{A}(\theta) = \bigg( \sum_{j \in \mathbb{N}} a_{ij}\Theta_j + 2\pi\mathbb{Z} \bigg)_{i \in \mathbb{N}}.
	\end{equation}
\end{prop}

\begin{proof}
It is immediate to verify that any $A \in \mathrm{FGL}_{\mathbb{N}}(\mathbb{Z})$ defines, through \eqref{isomatrix} and \eqref{isomatrix2}, an automorphism of $\mathbb{T}^\mathbb{N}$.

Let us prove the converse. Given $\mathscr{A} \in \mathrm{Aut}(\mathbb{T}^\mathbb{N})$, we first recall the identification $\mathbb{T}^\mathbb{N} \cong \mathbb{R}^\mathbb{N}/(2\pi\mathbb{Z})^\mathbb{N}$ established in Proposition \ref{quoz}. By Theorem \ref{lifts}, there exists a unique continuous lift $\hat{\mathscr{A}}: \mathbb{R}^\mathbb{N} \to \mathbb{R}^\mathbb{N}$ of $\mathscr{A}$ such that $\hat{\mathscr{A}}(0)=0$ and $\pi_\mathbb{N} \circ \hat{\mathscr{A}} = \mathscr{A} \circ \pi_\mathbb{N}$. Since $\mathbb{R}^\mathbb{N}$ is connected and $\mathscr{A}$ is an endomorphism, $\hat{\mathscr{A}}$ is an endomorphism as well; its continuity then ensures that it is an $\mathbb{R}$-linear map.

Let $e_j$ be the $j$-th canonical vector in $\mathbb{R}^\mathbb{N}$. For any $\Theta \in \mathbb{R}^\mathbb{N}$, let $(\Theta^n)_{n \in \mathbb{N}}$, where $\Theta^n := \sum_{j=1}^n \Theta_j e_j$, be the sequence of its truncations. By the definition of the product topology, $\Theta^n \to \Theta$ as $n \to \infty$. Since $\hat{\mathscr{A}}$ is continuous, it preserves the limits of convergent sequences; thus, for each component $i \in \mathbb{N}$, we have
\begin{equation}\label{limit}
(\hat{\mathscr{A}}(\Theta))_i = \left(\hat{\mathscr{A}}\left(\lim_{n \to \infty} \Theta^n\right)\right)_i = \lim_{n \to \infty} (\hat{\mathscr{A}}(\Theta^n))_i.
\end{equation}
By the linearity of $\hat{\mathscr{A}}$ and the fact that $\Theta^n$ has finite support, it follows that, for any $n \in \mathbb{N}$, $(\hat{\mathscr{A}}(\Theta^n))_i = \sum_{j=1}^n a_{ij}\Theta_j$, where $a_{ij} := (\hat{\mathscr{A}}(e_j))_i$. Substituting this into \eqref{limit}, we obtain $(\hat{\mathscr{A}}(\Theta))_i = \sum_{j \in \mathbb{N}} a_{ij}\Theta_j.$

The convergence of this series for every $\Theta \in \mathbb{R}^\mathbb{N}$ implies that each row of $A$ must have finite support. Indeed, if for some $i$ there were infinitely many non-zero entries $(a_{ij_k})_{k \in \mathbb{N}}$, one could define $\Theta \in \mathbb{R}^\mathbb{N}$ by setting $\Theta_{j_k} = k/a_{ij_k}$ (and $\Theta_j = 0$ otherwise), making the $i$-th component of $\hat{\mathscr{A}}(\Theta)$ diverge, which is a contradiction.

Furthermore, the relation $\pi_\mathbb{N} \circ \hat{\mathscr{A}} = \mathscr{A} \circ \pi_\mathbb{N}$ requires that $\hat{\mathscr{A}}$ maps the kernel of the projection $\pi_\mathbb{N}$, namely $(2\pi\mathbb{Z})^\mathbb{N}$, into itself; hence $a_{ij} \in \mathbb{Z}$ for all $i,j \in \mathbb{N}$. Finally, since $\mathscr{A}$ is an automorphism, there exists a continuous inverse $\mathscr{A}^{-1}$ which, by the same arguments, possesses a unique continuous linear lift $(\widehat{\mathscr{A}^{-1}})$ represented by an integer matrix $A'$ with finite support rows. By the uniqueness of lifts, the relations $\mathscr{A} \circ \mathscr{A}^{-1} = \mathrm{Id}_{\mathbb{T}^\mathbb{N}}$ and $\mathscr{A}^{-1} \circ \mathscr{A} = \mathrm{Id}_{\mathbb{T}^\mathbb{N}}$ imply that $\hat{\mathscr{A}} \circ (\widehat{\mathscr{A}^{-1}}) = \mathrm{Id}_{\mathbb{R}^\mathbb{N}}$ and $(\widehat{\mathscr{A}^{-1}}) \circ \hat{\mathscr{A}} = \mathrm{Id}_{\mathbb{R}^\mathbb{N}}$, respectively. In terms of matrices, this yields $AA' = A'A = I$, confirming that $A \in \mathrm{FGL}_{\mathbb{N}}(\mathbb{Z})$.
\end{proof}

\paraga \textit{Reduction.} Let $\mathrm{Aut}(\Z^{(\N)})$ be the group of bijective maps $\mathscr{B}$ from $\Z^{(\N)}$ to itself such that both $\mathscr{B}$ and $\mathscr{B}^{-1}$ are homomorphisms of $\Z$-modules.

\begin{ex}[Permutations in $\mathrm{Aut}(\mathbb{Z}^{(\mathbb{N})})$]
	Let $\sigma: \mathbb{N} \to \mathbb{N}$ be a bijection. We define the operator $P_{\sigma}$ by its action on the basis: $P_{\sigma}(e_j) = e_{\sigma(j)}$, extending it by linearity to any $\nu = \sum \nu_j e_j \in \mathbb{Z}^{(\mathbb{N})}$.
	
	Since any $\nu \in \mathbb{Z}^{(\mathbb{N})}$ has finite support, its image $P_\sigma(\nu)$ also has finite support, making $P_\sigma$ a well-defined $\mathbb{Z}$-module homomorphism from $\mathbb{Z}^{(\mathbb{N})}$ into itself. The associated matrix $(b_{ij})_{i,j \in \N}$ is:
	\begin{equation*}
	b_{ij} = 
	\begin{cases} 
	1 & \text{if } i = \sigma(j), \\
	0 & \text{otherwise.}
	\end{cases}
	\end{equation*}
	Since $\sigma$ is a bijection, each row $i$ has exactly one non-zero entry at column $\sigma^{-1}(i)$. This ensures that the inverse operator $P_{\sigma^{-1}}$ is also a well-defined homomorphism on $\mathbb{Z}^{(\mathbb{N})}$, thus $P_{\sigma} \in \mathrm{Aut}(\mathbb{Z}^{(\mathbb{N})})$.
\end{ex}

\begin{prop}\label{autmodule}
	An operator $\mathscr{B}$ belongs to $\mathrm{Aut}(\mathbb{Z}^{(\mathbb{N})})$ if and only if its representing matrix $B = (b_{ij})_{i,j \in \mathbb{N}}$ with respect to the canonical basis $\{e_j\}_{j \in \mathbb{N}}$ satisfies $B^* \in \mathrm{FGL}_{\mathbb{N}}(\mathbb{Z})$, where $B^*$ denotes the transpose matrix of $B$.
\end{prop}

\begin{proof}
$(\implies)$ Let $\mathscr{B} \in \mathrm{Aut}(\mathbb{Z}^{(\mathbb{N})})$. As a $\mathbb{Z}$-module homomorphism, $\mathscr{B}$ is uniquely determined by its action on the canonical basis $\{e_j\}_{j \in \mathbb{N}}$. Since $\mathscr{B}(e_j) \in \mathbb{Z}^{(\mathbb{N})}$ for all $j \in \mathbb{N}$, we can write its image as a finite linear combination of the basis vectors:
\begin{equation}\label{basis_action}
\mathscr{B}(e_j) = \sum_{i \in \mathbb{N}} b_{ij} e_i,
\end{equation}
where $b_{ij} \in \mathbb{Z}$ are zero for all but finitely many $i$. This defines the representing matrix $B = (b_{ij})_{i,j \in \mathbb{N}}$, which has integer entries and columns with finite support. Furthermore, as $\mathscr{B}$ is an automorphism, there exists an inverse homomorphism $\mathscr{B}^{-1}$ on $\mathbb{Z}^{(\mathbb{N})}$ whose matrix $B^{-1}=(b'_{jk})_{j,k \in \mathbb{N}}$ must also have integer entries and columns with finite support. 
The infinite matrix products $B B^{-1}$ and $B^{-1} B$ are well-defined because each entry is a sum with only finitely many non-zero terms: specifically, $(B B^{-1})_{ik} = \sum_{j \in \mathbb{N}} b_{ij} b'_{jk}$ involves only finitely many terms due to the finite support of the columns of $B^{-1}$, while $(B^{-1} B)_{ik} = \sum_{j \in \mathbb{N}} b'_{ij} b_{jk}$ is finite due to the finite support of the columns of $B$. This property ensures that the standard algebraic identity for the transpose of a product $(B B^{-1})^* = (B^{-1})^* B^*$ holds. Thus, transposing $B B^{-1} = B^{-1} B = I$ yields $(B^{-1})^* B^* = B^* (B^{-1})^* = I$, showing that $B^*$ is invertible in the sense of $\mathrm{FGL}_{\mathbb{N}}(\mathbb{Z})$ with $(B^*)^{-1} = (B^{-1})^*$. Since the rows of $B^*$ and $(B^*)^{-1}$ are the columns of $B$ and $B^{-1}$ respectively, we have $B^* \in \mathrm{FGL}_{\mathbb{N}}(\mathbb{Z})$.
	
	$(\impliedby)$ Conversely, suppose $B^* \in \mathrm{FGL}_{\mathbb{N}}(\mathbb{Z})$. By definition, both $B^*$ and its inverse $(B^*)^{-1}$ have integer entries and rows with finite support. This ensures that $B=(b_{ij})_{i,j \in \N}$ and the matrix $B' := ((B^*)^{-1})^*$ have integer entries and columns with finite support.
	The finite support of the columns of $B$, together with the fact that $b_{ij} \in \Z$ for all $i,j \in \N$, ensures that the linear extension $\mathscr{B}$ of the map $e_j \mapsto \sum_i b_{ij}e_i$ is a well-defined homomorphism from $\mathbb{Z}^{(\mathbb{N})}$ into itself; similarly, the operator $\mathscr{B}'$ associated to $B'$ is a well-defined homomorphism on $\mathbb{Z}^{(\mathbb{N})}$. Again, the finite support of the columns of $B$ and $B'$ justifies the identity $(B' B)^* = B^* (B')^* = B^* (B^*)^{-1} = I$, which implies $B' B = I$. A symmetric argument shows $B B' = I$, proving that $\mathscr{B}'$ is the inverse of $\mathscr{B}$ and thus $\mathscr{B} \in \mathrm{Aut}(\mathbb{Z}^{(\mathbb{N})})$.
\end{proof}

\begin{lemma}\label{abstractreduction}
	Let $\nu \in \mathbb{Z}^{(\mathbb{N})}\setminus \{0\}$. Then, there exists $\mathscr{B} \in \mathrm{Aut}(\mathbb{Z}^{(\mathbb{N})})$ such that $\tilde{\nu}:=\mathscr{B}\nu$ satisfies $\tilde{\nu}_1\neq 0$ and $\tilde{\nu}_j =0$ for all $j \geq 2$.
\end{lemma}

\begin{proof}
	If the support of $\nu$ consists of a single element, say $j^*$, then it is sufficient to consider the permutation operator $P_\sigma$ (as in Example 1.1) that swaps the first and the $j^{*}$-th component. By Proposition \ref{autmodule}, $P_\sigma \in \mathrm{Aut}(\mathbb{Z}^{(\mathbb{N})})$. 
	
	Consider now the case where the cardinality of the support of $\nu$ is $N>1$. Let $B_1$ be the matrix that assigns to the given $\nu$ the list $(|\nu_{j_1}|, \dots, |\nu_{j_N}|, 0, \dots)$, where indices are chosen such that $1 \leq |\nu_{j_1}| \leq \cdots \leq |\nu_{j_{N}}|$. This matrix represents an automorphism of $\mathbb{Z}^{(\mathbb{N})}$ as it is a finite composition of permutations $P_\sigma$ and a diagonal matrix with $\pm 1$ on the diagonal. Define $\nu^1 \in \mathbb{Z}^{(\mathbb{N})}$ by setting $\nu^1_k = |\nu_{j_k}|$ for $k=1, \dots, N$ and $\nu^1_k = 0$ for $k > N$, and let 
	\begin{equation}
	S^1 := \sum_{j=1}^{N} \nu_j^1.
	\end{equation}
	Consider the transformation $\nu^1 \mapsto \nu^2$ defined by $\nu^2_1 = \nu^1_1$, $\nu^2_j = \nu^1_j - \nu^1_1$ for $2 \leq j \leq N$, and $\nu^2_j = 0$ for $j > N$. This transformation is induced by a matrix $B_2$ representing an automorphism of $\mathbb{Z}^{(\mathbb{N})}$:
	\begin{equation}
	B_2=\begin{pmatrix}
	\ \ 1 & 0 & \cdots & \cdots & \cdots \ \  \\
	-1 & 1 & 0 & \cdots & \cdots \ \  & \\
	\ \ \vdots & 0 & \ddots \ \  \\
	-1 & \vdots & & \ddots&    \\
	\ \ 0 & \vdots & & & \ddots  \\
	\ \ \vdots & \vdots
	\end{pmatrix}
	\ \ \ 
	\ \ \ B^{-1}_2=\begin{pmatrix}
	\ \  1 & 0 & \cdots & \cdots & \cdots \ \ \  \\
	\ \ 1 & 1 & 0 & \cdots & \cdots \ \ \  \\
	\ \ \ \vdots & 0 & \ddots \ \ \  \\
	\ \ 1 & \vdots &  & \ddots \ \ \     \\
	\ \ 0 & \vdots & & & \ddots \ \ \  \\
	\ \ \vdots & \vdots
	\end{pmatrix}
	\end{equation}
	
	Note that $\nu^2_1 = \nu^1_1 \geq 1$ and, for $j=2, \dots, N$, we have $0 \leq \nu^2_j < \nu^1_j$. Thus, if we define $S^2 := \sum_{j=1}^{N} \nu_j^2$, we obtain $1 \leq S^2 < S^1$. 
	
	At this stage, two cases may occur:
	\begin{enumerate}
		\item If $\nu^2$ has only one non-zero component, then by construction this component is $\nu^2_1$ and the process terminates.
		\item If $\nu^2$ still has more than one non-zero component, we denote by $K \geq 0$ the number of components that have become zero (i.e., $\nu^2_j = 0$ for some $j \in \{2, \dots, N\}$). We then apply a permutation $P_{\sigma}$ to move the $N-K$ remaining non-zero components to the first $N-K$ positions. 
	\end{enumerate}
	
	After this rearrangement, we reorder the remaining non-zero components in non-decreasing order and repeat the entire procedure. Let $S^\alpha$ denote the sum of the components of the resulting vector at the end of the $\alpha$-th step. By iterating the process, we generate a strictly decreasing sequence of positive integers:
	\begin{equation}
	S^1 > S^2 > S^3 > \dots > S^\alpha > \dots \geq 1.
	\end{equation}
	This procedure must terminate in a finite number of steps $\bar{\alpha}$ because a strictly decreasing sequence of positive integers cannot be infinite. Specifically, the process stops when the sum $S^{\bar{\alpha}}$ consists of only one non-zero term.
	
	By construction, this remaining term is moved to the first position. Since the overall transformation $\mathscr{B}$ is a composition of a finite number of automorphisms, it follows from Proposition \ref{autmodule} and the group property of $\mathrm{Aut}(\mathbb{Z}^{(\mathbb{N})})$ that $\mathscr{B} \in \mathrm{Aut}(\mathbb{Z}^{(\mathbb{N})})$, which completes the proof.
\end{proof}

\begin{lemma}\label{reduction}
	Given $\om \in \R^\N$, if there exists $ \nu \in \Z^{(\N)}$ such that $\sum_{j\in\N}\om_j\nu_j=0$, then, there is an $\tilde{\om}\in\R^\N$, with $\tilde{\om}_1=0$,  such that the Kronecker flow $(\T^\N,\Phi_\om)$ is topologically conjugate to $(\T^\N,\Phi_{\tilde{\om}})$.
\end{lemma}

\begin{proof}
	Let $\nu \in \Z^{(\N)}\setminus \{0\}$ be such that $\sum_{j \in \N}\om_j\nu_j=0$. By Lemma \ref{abstractreduction}, there exists $\mathscr{B}\in\mathrm{Aut}(\Z^{(\N)})$ such that $\tilde{\nu}:=\mathscr{B}\nu$ satisfies $\tilde{\nu}_1 \neq 0$ and $\tilde{\nu}_j=0$ for all $j>1$. Remark \ref{autmodule} tells us that such a $\mathscr{B}$ is represented by a matrix $B$ such that $B^* \in \mathrm{FGL}_{\N}(\Z)$. By Proposition \ref{matrix}, there exists an automorphism $\mathscr{A}$ of $\T^\N$ that is represented by $B^*$. But then, the Kronecker flows $(\T^\N,\Phi_{\tilde{\om}})$, with $\tilde{\om}:=B^*\tilde{\om}$, and $(\T^\N,\Phi_\om)$ are topologically conjugate, and $\tilde{\om}_1=0$.
\end{proof}

\begin{prop}\label{propprel}
	Let $\om \in \mathbb{R}^\N$ be such that $1\leq d:=\mathrm{rank}(\mathscr{R}_\om)<\infty$. Then, $(\T^\N,\Phi_\om)$ is topologically conjugate to another Kronecker flow $(\T^\N,\Phi_{\tilde{\om}})$ with $\tilde{\om}=(0_d,\bar{\om})$, where $\bar{\om}$ is non-resonant and $0_d$ is the null vector in $\R^d$.
	\end{prop}

	\begin{proof}
    By the above lemma, a resonant Kronecker flow $(\T^\N,\Phi_\om)$ is topologically conjugate to another Kronecker flow such that the first component of the frequency vector vanishes. If the restriction to the support of such a new frequency vector is still resonant, we get rid of another frequency following the same procedure as before. If the dimension of the resonance module is finite, this process will end after a number of iteration equal to the dimension of $\mathscr{R}_\om$, because it is sufficient to do it for a basis of $\mathscr{R}_\om$. The resulting frequency vector $\tilde{\om}$ will be such that its non-vanishing frequencies are rationally independent by construction.
	\end{proof}
	
	\section{The Non-Resonant Case}\label{ue}
	\paraga \textit{The Haar Measure.} Let $\mathscr{P}(\T^\N,\mathbb{R})$ be the space of trigonometric polynomials on $\T^\N$, i.e.,
	\begin{equation*}
	\mathscr{P}(\T^\N,\mathbb{R}):=\bigg\{p:\T^\N \to \mathbb{R} \; : \; p(\theta)=\sum_{\nu \in \Z^{(\N)}}a_\nu e^{\im \Theta\cdot\nu}, \; (a_\nu)_{\nu \in \Z^{(\N)}} \in \bigoplus_{\nu\in \Z^{(\N)}}\C, \; a_{\nu}=\overline{a_{-\nu}}\bigg\}.
	\end{equation*}
	Consider the linear functional
	\begin{equation}\label{tilde}
	\tilde{\Lambda}:\mathscr{P}(\T^\N,\mathbb{R})\ni p \mapsto \lim_{N \to \infty}\dfrac{1}{(2\pi)^N}\int_{\T^N}p(\Theta) \, d\Theta_1 \dots d\Theta_{N}.
	\end{equation}
	It is easy to verify that the limit exists for any trigonometric polynomial. In fact, given $\nu \in \Z^{(\N)}$, there exists $M<\infty$ such that $e^{\im\Theta\cdot \nu}=e^{\im\nu_1\Theta_1}e^{\im \nu_2\Theta_1}\cdot \cdot \cdot e^{\im\nu_{M}\Theta_{M}}$. Thus, for $N\geq M$,
	\begin{equation}
	\begin{aligned}
	\lim_{N\to \infty}\dfrac{1}{(2\pi)^N}\int_{\T^N}e^{\im\Theta\cdot \nu} \, d\Theta_1 \dots d\Theta_{N}=\lim_{N \to \infty}\prod_{j=1}^{M}\int_{0}^{2\pi}e^{ \im\nu_j\Theta_j} \, \dfrac{d\Theta_j}{2\pi}\prod_{k=M+1}^{N}\int_0^{2\pi} \, \dfrac{d\Theta_{k}}{2\pi}.
	\end{aligned}
	\end{equation}
	The right-hand side of this equality can only take the values $0$, if there is $\nu_j\neq 0$, or $1$, if $\nu=0$. In particular, denoting by $\mathds{1}$ the trigonometric polynomial identically equal to one, we have
	\begin{equation}\label{prob}
	\tilde{\Lambda}(\mathds{1})=1.
	\end{equation} 
	Furthermore, $\tilde{\Lambda}$ is Lipschitz relative to the uniform norm, i.e.
	\begin{equation}
		\bigg|\lim_{N \to \infty}\dfrac{1}{(2\pi)^N}\int_{\T^N}p(\Theta) \, d\Theta_1 \dots d\Theta_{N}\bigg|\leq \|p\|_{\infty}\cdot 1.
	\end{equation}
 Since $\mathscr{P}(\T^\N,\mathbb{R})$ is a unital subalgebra of $\mathscr{C}^0(\T^\N,\mathbb{R})$ which separates the points of the compact space $\T^\N$, by the Stone-Weierstrass Theorem, each continuous function on $\T^\N$ is a uniform limit of some sequence of trigonometric polynomials. Then, $\tilde{\Lambda}$ extends uniquely to a Lipschitz (and hence, continuous) linear functional
 
\begin{equation}\label{functionalmeasure}
	\Lambda:\mathscr{C}^0(\T^\N,\mathbb{R}) \to \mathbb{R}, \quad \Lambda(f):=\lim_{n \to \infty}\tilde{\Lambda}(p_n),
		\end{equation}
		where $(p_n)_{n \in \N}$ is any sequence in $\mathscr{P}(\T^\N,\mathbb{R})$ converging  uniformly to $f$. Moreover, $\Lambda$ enjoys also the property well explicated in the proposition just below.
	\begin{prop}
		The linear functional $\Lambda$ is positive, meaning that, if $f\in \mathscr{C}^0(\T^\N,\mathbb{R})$ satisfies $\inf_{\theta\in\T^\N}f(\theta)\geq 0$, then $\Lambda(f)\geq 0$.
		\begin{proof}
			Let $(p_n)_{n \in \N}$ be any sequence in $\mathscr{P}(\T^\N,\mathbb{R})$ uniformly converging to $f$. Then,
			\begin{equation}
		    \begin{aligned}
			\Lambda(f)=\lim_{n \to \infty}\tilde{\Lambda}(p_n)&=\lim_{n \to \infty}\lim_{N \to \infty}\dfrac{1}{(2\pi)^N}\int_{\T^N}p_n(\Theta) \, d\Theta_1\dots d\Theta_{N}\\ &\geq \lim_{n\to\infty}\inf_{\theta \in \T^\N}p_n(\theta)=\inf_{\theta \in \T^\N}f(\theta)\geq 0 \ .
			\end{aligned}
			\end{equation}
		\end{proof}
	\end{prop}

	In what follows, we will repeatedly apply the following celebrated theorem.
	\begin{thm}[Riesz Representation Theorem \cite{rudin}]\label{Rieszthm}
		Let $X$ be a locally compact space, and let $\Lambda$ be a positive and continuous linear functional on $\mathscr{C}^0(X,\mathbb{R})$. Then there exists a $\sigma$-algebra $\mathfrak M$ on $X$ which contains all Borel subsets of $X$, and there exists a unique positive measure $\lambda$ on $\mathfrak M$ which represents $\Lambda$ in the sense that
		$$\Lambda f=\int_X f \, d\lambda, \quad \forall f\in \mathscr{C}^0(X,\mathbb{R}),$$ and which has the following additional properties:
		\begin{enumerate}
			\item[(a)] $\lambda(V):=\sup\{\Lambda(f) \; : \; \text{supp}(f)\subset V\}$ for any open set $V\subset X$;
			\item[(b)] $\lambda(K)<\infty \ $ for every compact set $K\subset X$;
			\item[(c)] 	$\lambda$ is inner and outer regular, meaning that for every $E\subset \mathfrak{M}$ we have
			\begin{equation*}
			\lambda(E)=\inf\{\lambda(V) \; : \; E\subset V, \; V \; \text{open}\}=\sup\{\lambda(K) \; : \; K\subset E, \; K \; \text{compact}\}.
			\end{equation*}
		\end{enumerate}
	
	\end{thm}
\noindent Hence, $\mathbb{T}^\mathbb{N}$ admits a $\sigma$-algebra $\mathfrak{M}$ containing all Borel sets such that there exists a unique positive measure $\lambda$ on $\mathfrak{M}$ associated to the linear functional $\Lambda$ (defined in \eqref{tilde} and \eqref{functionalmeasure}) in the sense of the above theorem. Thus, by restricting $\lambda$ to the Borel algebra $\mathfrak{B}$ and taking \eqref{prob} into account, we conclude that $\mu := \lambda_{|\mathfrak{B}}$ is the unique regular Borel probability measure on $\mathbb{T}^\mathbb{N}$ that represents $\Lambda$, i.e.,

\begin{equation}
\Lambda(f)=\int_{\mathbb{T}^\mathbb{N}}f \, d\mu, \quad \forall f \in \mathscr{C}^0(\mathbb{T}^\mathbb{N},\mathbb{R}).
\end{equation}

\begin{thm}
	Let $\theta' \in \mathbb{T}^\mathbb{N}$ and let $R_{\theta'}: \mathbb{T}^\mathbb{N} \to \mathbb{T}^\mathbb{N}$ be the translation defined by $R_{\theta'}(\theta) = \theta * \theta'$, for all $\theta \in \T^\N$. The unique Borel measure $\mu$ associated to $\Lambda$ (as defined in \eqref{functionalmeasure}) in the sense of Theorem \ref{Rieszthm} is $R_{\theta'}$-invariant.
\end{thm}

\begin{proof}
	By Theorem \ref{Rieszthm}, specifically properties $(a)$ and $(c)$, it suffices to show that $\Lambda(f \circ R_{\theta'})=\Lambda(f)$ for all $f \in \mathscr{C}^0(\mathbb{T}^\mathbb{N}, \mathbb{R})$. 
	First, observe that this equality clearly holds for any trigonometric polynomial $p \in \mathscr{P}(\mathbb{T}^\mathbb{N}, \mathbb{R})$. 
	
	Now, let $f \in \mathscr{C}^0(\mathbb{T}^\mathbb{N}, \mathbb{R})$ and let $(p_n)_{n \in \mathbb{N}}$ be a sequence of trigonometric polynomials in $\mathscr{P}(\mathbb{T}^\mathbb{N}, \mathbb{R})$ that converges uniformly to $f$. It follows that the sequence $(p_n \circ R_{\theta'})_{n \in \mathbb{N}}$ converges uniformly to $f \circ R_{\theta'}$. Consequently, by the continuity of the functional $\Lambda$, we have:
	
	\begin{equation}
	\Lambda(f \circ R_{\theta'}) = \lim_{n \to \infty} \Lambda(p_n \circ R_{\theta'}) = \lim_{n \to \infty} \Lambda(p_n) = \Lambda(f). 
	\end{equation}
\end{proof}

In particular, $\mu$ is $\Phi^t_\om$-invariant for all $\om \in \mathbb{R}^\N$ and any time $t$. Thus, \eqref{tilde} and \eqref{functionalmeasure} define a probability invariant Borel measure for any Kronecker flow $(\T^\N,\Phi_\om)$.
\begin{rmk}
	The measure just constructed is nothing but the Haar measure on the infinite torus, and corresponds to the invariant integral introduced by Jessen in \cite{Jess}. For references on the Haar measure, see for e.g. \cite{dieudonne}.
\end{rmk}
\paraga \textit{Proof of Theorem \ref{A}.} Let us first prove that, given $f \in \mathscr{C}^0(\T^\N,\R)$, its time average along any orbit of a non-resonant Kronecker flow converges uniformly to the integral of $f$ with respect to the Haar measure over the whole torus. This step is crucial for the proof of Theorem \ref{A}.
\begin{thm}\label{avecon}
		If $\mathscr{R}_\om=\{0\}$, then $\forall f \in \mathscr{C}^0(\T^\N,\mathbb{R})$
		\begin{equation}\label{time-stat average}
		\lim_{T\to \infty}\bigg\|\frac{1}{T}\int_{0}^{T}f \circ \Phi^t_\om \, dt -\int_{\T^\N}f \, d\mu\bigg\|_{\infty}=0.
		\end{equation}
\begin{proof}
	Fix $f \in \mathscr{C}^0(\T^\N,\R)$.  If $f$ is constant, then \eqref{time-stat average} is trivial. Moreover, $ \forall \nu \in \Z^{(\N)}\setminus\{0\}$,
	\begin{equation}
	\lim_{T\to \infty}\sup_{\theta\in \T^\N}\bigg| \frac{1}{T}\int_{0}^{T}e^{\im\nu \cdot (\Theta+\om t)} \, dt -\int_{\T^\N}e^{\im \nu \cdot \Theta} \, d\mu \bigg|=\lim_{T\to \infty}\sup_{\theta\in \T^\N}\bigg| \frac{1}{T}e^{\im\nu \cdot \Theta}\dfrac{e^{\im\nu \cdot \om T}-1}{\im\om \cdot \nu} \bigg|=0. 
	\end{equation}
	
	This is also true for any finite linear combination, i.e., for any trigonometric polynomial.
	Now, given $\eps>0$, let $p\in \mathscr P(\T^\N,\mathbb{R})$ be such that $\|f-p\|_\infty<\frac{\eps}{3}$ and let $T_{\frac{\eps}{3}}$ be such that
	\begin{equation}
	\bigg\|\frac{1}{T}\int_{0}^{T}p \circ \Phi^t_\om \, dt -\int_{\T^\N}p \, d\mu \bigg\|_{\infty}<\dfrac{\eps}{3} \quad \forall T>T_{\frac{\eps}{3}}.
	\end{equation} 
	Then,
	\begin{gather}
	\bigg\| \frac{1}{T}\int_{0}^{T}f \circ \Phi^t_\om \, dt -\int_{\T^\N}f \, d\mu \bigg\|_{\infty} \leq\bigg\| \frac{1}{T}\int_{0}^{T}(f-p) \circ \Phi^t_\om \, dt\bigg\|_{\infty}\\
	+\bigg\| \frac{1}{T}\int_{0}^{T}p \circ \Phi^t_\om \, dt -\int_{\T^\N}p \, d\mu \bigg\|_{\infty}+\bigg|\int_{\T^\N}(f-p) \, d\mu\bigg|\\
	\leq 2\|f-p\|_\infty+\bigg\|\frac{1}{T}\int_{0}^{T}p \circ \Phi^t_\om \, dt -\int_{\T^\N}p \, d\mu \bigg\|_{\infty} \ < \eps.
	\end{gather}
	Therefore, for all $\eps>0$, we have
	\begin{equation}
	\bigg\| \frac{1}{T}\int_{0}^{T}f \circ \Phi^t_\om \, dt -\int_{\T^\N}f \, d\mu \bigg\|_{\infty}<\eps \quad \forall T>T_\frac{\eps}{3}.
	\end{equation}
\end{proof}
	\end{thm}

\begin{thmA}\label{A}
The following properties of a Kronecker flow on $\T^\N$ are equivalent:
	\begin{itemize}
		\item[(a)] it is non-resonant;
		\item[(b)] it is uniquely ergodic;
		\item[(c)] it is minimal;
		\item[(d)] it is topologically transitive.
	\end{itemize}
\begin{proof}
	$(a)\Rightarrow (b)$. Let $\mu'$ be an invariant regular probability Borel measure for the Kronecker flow $(\T^\N,\Phi_\om)$. Then, for any continuous function $f:\T^\N \to \R$, one has
	\begin{equation}\label{fubini}
	\int_{\T^\N}f \, d\mu'=\lim_{T \to \infty} \dfrac{1}{T}\int_{0}^T\bigg(\int_{\T^\N}f \, d\mu'\bigg) \, dt = \lim_{T \to \infty} \dfrac{1}{T}\int_{0}^T\bigg(\int_{\T^\N}f\circ \Phi^t_\om \, d\mu'\bigg) \, dt \,.
	\end{equation}
	Since $|f\circ \Phi^t_\om|$ is a continuous function on the compact domain $[0,T] \times \T^\N$, its integral over the product Borel measure space $[0,T] \times \T^\N$ is finite. Therefore, by Fubini's Theorem, we can exchange the order of integration, namely, 
	\begin{equation}
	\eqref{fubini}=\lim_{T\to \infty}\int_{\T^\N}\frac{1}{T}\bigg(\int_{0}^Tf\circ\Phi^t_\om \, dt\bigg) \, d\mu'.
	\end{equation}
For all $\theta \in \mathbb{T}^\mathbb{N}$ and $T > 0$, we observe that
\begin{equation}\label{dominante}
\left| \frac{1}{T}\int_{0}^T f \circ \Phi^t_\omega(\theta) \, dt \right| \leq \sup_{\theta \in \mathbb{T}^\mathbb{N}} |f(\theta)| = \|f\|_{\infty}.
\end{equation}
Since $f \in \mathscr{C}^0(\mathbb{T}^\mathbb{N}, \mathbb{R})$ and $\mathbb{T}^\mathbb{N}$ is compact, the constant $\|f\|_{\infty}$ is finite and thus provides an integrable dominating function on the probability space $(\mathbb{T}^\mathbb{N}, \mu')$. 

Furthermore, if $\omega$ is non-resonant, Theorem \ref{avecon} ensures that the time average converges uniformly (and thus pointwise) to $\int_{\mathbb{T}^\mathbb{N}} f \, d\mu$ as $T \to \infty$. Since the map 
\begin{equation*}
T \mapsto \frac{1}{T} \int_{0}^T f \circ \Phi^t_\omega \, dt
\end{equation*}
is continuous and its limit as $T \to \infty$ exists, this limit is uniquely determined and coincides with the limit along any sequence $(T_n)_{n \in \mathbb{N}}$ such that $T_n \to \infty$ as $n \to \infty$. This allows us to invoke the Lebesgue's Dominated Convergence Theorem for any such sequence, effectively justifying the interchange of the limit and the integral over $\mathbb{T}^\mathbb{N}$. Therefore, we obtain:
\begin{equation}
\eqref{fubini} = \int_{\mathbb{T}^\mathbb{N}} \left( \lim_{T \to \infty} \frac{1}{T} \int_{0}^T f \circ \Phi^t_\omega \, dt \right) d\mu' = \int_{\mathbb{T}^\mathbb{N}} \left( \int_{\mathbb{T}^\mathbb{N}} f \, d\mu \right) d\mu'.
\end{equation}
	Using the fact that $\mu'$ is a probability measure, we finally get
	\begin{equation}
	\int_{\T^\N}f \, d\mu'=\int_{\T^\N}f \, d\mu \,.
	\end{equation}
	Since $f\in \mathscr{C}^0(\T^\N,\mathbb{R})$ is arbitrary, by Theorem \ref{Rieszthm}, we conclude that $\mu'=\mu$.
	\medskip
	
	\noindent $(b) \Rightarrow (a)$. Assume that there exists $\nu \in \Z^{(\N)}\setminus\{0\}$ such that $\om\cdot \nu=0$. By Lemma \ref{reduction}, $(\T^\N,\Phi_\om)$ is topologically conjugate to another Kronecker flow whose frequency vector $\tilde{\om}$ has the first component equal to zero. Given $x \in \T$, define the linear functional
	\begin{equation}
	\tilde{\Lambda}_{x}:\mathscr{P}(\T^\N,\mathbb{R})\ni p \mapsto \lim_{N \to \infty}\dfrac{1}{(2\pi)^{N-1}}\int_{\T^{N-1}}p(x,\Theta_2,\Theta_3, \dots) \, d\Theta_2 \dots d\Theta_{N}.
	\end{equation}
As in the case discussed in the introduction to this section, all these functionals extend to positive linear functionals $\Lambda_x$ on the entire space $\mathscr{C}^0(\mathbb{T}^\mathbb{N}, \mathbb{R})$. To each of them, there is associated a unique regular Borel probability measure $\mu_x$ in the sense of Theorem \ref{Rieszthm}. 

Moreover, these measures are all invariant under the Kronecker flow $(\mathbb{T}^\mathbb{N}, \Phi_{\tilde{\omega}})$. Let $\mathscr{A}$ be the automorphism of $\mathbb{T}^\mathbb{N}$ that conjugates $(\mathbb{T}^\mathbb{N}, \Phi_\omega)$ to $(\mathbb{T}^\mathbb{N}, \Phi_{\tilde{\omega}})$. Then, for any $x \in \mathbb{T}$, the pullback probability Borel measures $\mathscr{A}^*\mu_x$ are invariant under $(\mathbb{T}^\mathbb{N}, \Phi_\omega)$. Clearly, for distinct values of $x \in \mathbb{T}$, the resulting pullback measures $\mathscr{A}^*\mu_x$ are distinct, thereby proving the existence of multiple invariant probability measures for the flow $(\mathbb{T}^\mathbb{N}, \Phi_\omega)$.
	\medskip
	
	\noindent $(a)\Rightarrow(c)$. Let $\om\in \R^\N$ be non-resonant. Suppose that there exists a non-empty open set $V\subset \T^\N$ that is never visited by a certain orbit. Then the statement of Theorem \ref{avecon} would be violated, for instance, by any continuous function $f$ which is strictly positive on a compact subset $K\subset V$ and such that $f|_{(\T^\N\setminus V)}\equiv 0$; such a function exists by Urysohn's Lemma.
	\medskip
	
	\noindent $(c)\Rightarrow(a)$. We prove the contrapositive. As in the proof of $(b)\Rightarrow(a)$, if a Kronecker flow is resonant, it is topologically conjugate through a certain $\mathscr{A}\in\mathrm{Aut}(\T^\N)$ to another Kronecker flow whose frequency vector has the first component equal to zero. Let $\theta \in \T^\N$. Then, the open set $\mathscr{A}^{-1}(C)$, where $C$ is the open cylinder 
	\begin{equation}
	C:=(\T\setminus\{\Theta_1+2\pi\Z\})\times\prod_{j=2}^{\infty}\T\subset \T^\N,
	\end{equation}
	is never visited by the orbit of $\mathscr{A}^{-1}(\theta)$.
	\medskip
	
	\noindent $(d)\Rightarrow(c)$. It is an immediate consequence of Remark \ref{orb-Oom}.
	\medskip
	
	\noindent $(c)\Rightarrow(d)$. It follows trivially by definition.
\end{proof}
\end{thmA}

\section{The Resonant Case: an Algebraic Viewpoint}\label{ccr}
\subsection{The \textquotedblleft Free" Case}\label{thefreecase} The main goal is to give a proof of Theorem \ref{B}, which gives a dynamical characterization of the class of Kronecker flows whose frequency module is a free abelian group. 

\begin{notation}[Closure of the orbit through the origin]
	Given a frequency vector $\omega \in \mathbb{R}^{\mathbb{N}}$, we denote by $\mathcal{O}_{\omega}$ the compact subgroup of $\mathbb{T}^\mathbb{N}$ obtained by taking the closure of the orbit through the origin induced by the Kronecker flow $(\mathbb{T}^\mathbb{N}, \Phi_\omega)$:
	\begin{equation}
	\mathcal{O}_\omega := \overline{\{ \Phi_\omega^t(0) \; : \; t \in \mathbb{R} \}}.
	\end{equation}
\end{notation}

\begin{rmk}
	To see that $\mathcal{O}_\omega$ is a subgroup, let $E = \{ \Phi_\omega^t(0)  : t \in \mathbb{R} \}$. Since the map $t \mapsto \Phi_\omega^t(0)$ is a continuous group homomorphism, its image $E$ is a subgroup of $\mathbb{T}^\mathbb{N}$. Now, let $\theta, \varphi \in \mathcal{O}_\omega$. By definition of closure, there exist sequences $(\theta_n)_{n \in \N}, (\varphi_n)_{n \in \N} \subset E$ such that $\theta_n \to \theta$ and $\varphi_n \to \varphi$ as $n\to \infty$. Due to the continuity of the group operations in $\mathbb{T}^\mathbb{N}$, we have:
	\begin{enumerate}
\item $\theta_n * \varphi_n \to \theta * \varphi$. The fact that $E$ is a subgroup ensures $(\theta_n * \varphi_n) \in E$ for all $n \in \N$, so $(\theta * \varphi) \in \overline{E} = \mathcal{O}_\omega$.
		\item $-\theta_n \to -\theta$. Since $E$ is a subgroup, $-\theta_n \in E$ for all $n \in \N$, hence $-\theta \in \overline{E} = \mathcal{O}_\omega$.
	\end{enumerate}
	Thus, $\mathcal{O}_\omega$ is a subgroup. Its compactness follows from being a closed subset of the compact space $\mathbb{T}^\mathbb{N}$.
\end{rmk}

\begin{lemma}\label{lemmaB}
	Let $\omega, \tilde{\omega} \in \mathbb{R}^{\mathbb{N}}$ be two frequency vectors. Assume that:
	\begin{itemize}
		\item[(i)] the Kronecker flows $(\mathbb{T}^\mathbb{N}, \Phi_\omega)$ and $(\mathbb{T}^\mathbb{N}, \Phi_{\tilde{\omega}})$ are topologically conjugate via a homeomorphism $h: \mathbb{T}^\mathbb{N} \to \mathbb{T}^\mathbb{N}$ such that $h(0)=0$;
		\item[(ii)] the non-vanishing components $(\tilde{\omega}_j)_{j \in J}$ are rationally independent, where $J \subseteq \mathbb{N}$ is the support of $\tilde{\omega}$.
	\end{itemize}
	Then, $\mathscr{M}_\omega = \mathscr{M}_{\tilde{\omega}}$.
\end{lemma}

\begin{proof}
	By the conjugacy relation $h \circ \Phi^t_{\tilde{\omega}} = \Phi^t_\omega \circ h$, $h$ maps the orbit of the origin of the first flow into the orbit of the origin of the second. Since the infinite torus is compact, the continuity of $h$ implies $h(\mathcal{O}_{\tilde{\omega}}) = \mathcal{O}_\omega$.
	
	We first prove that the restriction $h|_{\mathcal{O}_{\tilde{\omega}}}$ is a group isomorphism. Let $\theta \in \mathcal{O}_{\tilde{\omega}}$; by definition of the closure of the orbit, there exists a sequence $(t_k)_{k \in \mathbb{N}} \subset \mathbb{R}$ such that $\theta = \lim_{k \to \infty} \Phi^{t_k}_{\tilde{\omega}}(0)$. Now, let $\varphi \in \mathcal{O}_{\tilde{\omega}}$ be any other point. Using the conjugacy and the fact that $\Phi^t(\varphi) = \Phi^t(0) * \varphi$, we compute:
	\begin{align*}
	h(\theta * \varphi) &= h\left(\lim_{k \to \infty} \Phi^{t_k}_{\tilde{\omega}}(0) * \varphi\right) = \lim_{k \to \infty} h(\Phi^{t_k}_{\tilde{\omega}}(\varphi)) \\
	&= \lim_{k \to \infty} \Phi^{t_k}_{\omega}(h(\varphi)) = \lim_{k \to \infty} (\Phi^{t_k}_{\omega}(0) * h(\varphi)) \\
	&= \left(\lim_{k \to \infty} h(\Phi^{t_k}_{\tilde{\omega}}(0))\right) * h(\varphi) \\
	&= h(\theta) * h(\varphi).
	\end{align*}
	Thus, $h|_{\mathcal{O}_{\tilde{\omega}}}$ is an injective group homomorphism. Since $h$ is a homeomorphism and $h(\mathcal{O}_{\tilde{\omega}}) = \mathcal{O}_\omega$, the same argument applied to $h^{-1}$ shows that $h|_{\mathcal{O}_{\tilde{\omega}}}$ is an isomorphism of topological groups between $\mathcal{O}_{\tilde{\omega}}$ and $\mathcal{O}_\omega$.
	
	Now, let $\hat{h}: \mathbb{R}^\mathbb{N} \to \mathbb{R}^\mathbb{N}$ be the unique lift of $h$ that fixes the origin (see Theorem \ref{lifts}). Since $h|_{\mathcal{O}_{\tilde{\omega}}}$ is a continuous homomorphism and $\mathcal{O}_{\tilde{\omega}} = \pi_\mathbb{N}(E)$ for $E = \prod_{j \in \mathbb{N}} X_j$ (where $X_j = \mathbb{R}$ if $j \in J$ and $\{0\}$ otherwise), the restriction $\hat{h}|_E$ is linear. From the conjugacy, we obtain $\hat{h}(\tilde{\omega}) = \omega$.
	
	Since $h$ is a continuous map on $\mathbb{T}^\mathbb{N}$ that restricts to a homomorphism on $\mathcal{O}_{\tilde{\omega}}$, its lift $\hat{h}$ on $E$ is represented by an integer matrix $A=(a_{ij})_{i\in\mathbb{N}, j\in J}$ where the continuity in the product topology ensures that each row has finite support (this follows by an argument analogous to that of Proposition \ref{matrix}). Thus, $\omega_i = \sum_{j \in J} a_{ij} \tilde{\omega}_j$ for all $i \in \mathbb{N}$. Similarly, for $h^{-1}$, there exists an integer matrix $B=(b_{ij})_{i \in J,j \in \N}$ with rows of finite support such that $\tilde{\omega}_i = \sum_{j \in \mathbb{N}} b_{ij} \omega_j$ for all $i \in J$.
	
	Substituting $\omega_j$ into the expression for $\tilde{\omega}_i$ ($i \in J$):
	\[ \tilde{\omega}_i = \sum_{j \in \mathbb{N}} b_{ij} \sum_{k \in J} a_{jk} \tilde{\omega}_k \implies \sum_{k \in J} \left( \sum_{j \in \mathbb{N}} b_{ij} a_{jk} - \delta_{ik} \right) \tilde{\omega}_k = 0. \]
	By the rational independence of $(\tilde{\omega}_k)_{k \in J}$, we have $\sum_{j \in \mathbb{N}} b_{ij} a_{jk} = \delta_{ik}$. This implies that for any $\tilde{\nu} \in \mathbb{Z}^{(J)}$ there exists $\nu \in \mathbb{Z}^{(\mathbb{N})}$ such that $\tilde{\nu}_j = \sum_{i \in \mathbb{N}} \nu_i a_{ij}$ (specifically, $\nu_i = \sum_{j \in J} \tilde{\nu}_j b_{ji}$). Consequently:
	\begin{equation}
	\mathscr{M}_\omega = \left\{ \sum_{i \in \mathbb{N}} \omega_i \nu_i = \sum_{j \in J} \tilde{\omega}_j \sum_{i \in \mathbb{N}} a_{ij} \nu_i \; : \; \nu \in \mathbb{Z}^{(\mathbb{N})} \right\} = \left\{ \sum_{j \in J} \tilde{\omega}_j \tilde{\nu}_j \; : \; \tilde{\nu} \in \mathbb{Z}^{(J)} \right\} = \mathscr{M}_{\tilde{\omega}} .
	\end{equation}
	This concludes the proof.
\end{proof}

Note that one can give a shorter and less cumbersome proof of the above lemma by invoking Pontryagin's Duality Theorem.

\begin{rmk}\label{remsupp}
From an algebraic point of view, for a Kronecker flow the property of having non-zero frequencies that are rationally independent implies that the frequency module is free. Indeed, in that case, the map
\begin{equation}
\bigoplus_{j\in J}\Z\ni \nu \to \sum_{j\in J}\om_j\nu_j \in \mathscr{M}_\om,
\end{equation}
where $J\subseteq \N$ is the support of $\om\in\R^\N$, is an isomorphism of abelian groups.
\end{rmk}

\noindent The observation above inspires the following

\begin{thmB}\label{B} 
Let $\om \in \R^\N$ and let $(\T^\N,\Phi_\om)$ be the corresponding Kronecker flow. Then, $\mathscr{M}_{\om}$ is a free abelian group if and only if $(\T^\N,\Phi_\om)$ is topologically conjugate to another Kronecker flow $(\T^\N,\Phi_{\tilde{\om}})$ whose non-vanishing frequencies are rationally independent. If this is the case, $\mathscr{M}_{\om}$ is isomorphic to $\Z^{(\text{supp}(\tilde{\om}))}$.
\end{thmB}
\begin{proof}
$(\implies)$ Let $(\nu^{(\alpha)})_{\alpha \in \Lambda}$, $\Lambda\subseteq \N$ be a set of elements of $\Z^{(\N)}$ such that
\begin{equation}
\bigg(\sum_{j\in \N}\om_j\nu^{(\alpha)}_j\bigg)_{\alpha \in \Lambda}
\end{equation}
is a basis for $\mathscr{M}_\om$. Then, $\Z^{(\N)}$ is isomorphic to
\begin{equation}\label{dirsum}
\mathscr{R}_\om\bigoplus\mathrm{span}_{\Z}(\nu^{(\alpha)})_{\alpha \in \Lambda}.
\end{equation}
Indeed, since $\mathscr{M}_\om$ is isomorphic to the group of cosets $\Z^{(\N)}/\mathscr{R}_\om$ (Remark \ref{firstisothm}), for any $\nu \in \Z^{(\N)}$, there exists a unique set of natural numbers $(n_\alpha)_{\alpha\in \Lambda}$ such that $\nu+\mathscr{R}_\om=\sum_{\alpha \in\Lambda}n_\alpha (\nu^{(\alpha)}+\mathscr{R}_\om)$, where only finitely many $n_\alpha$ are different from zero. Then, every $\nu \in \Z^{(\N)}$ can be decomposed into  $\nu=(\nu-\sum_{\alpha \in\Lambda}n_\alpha \nu^{(\alpha)})+\sum_{\alpha \in\Lambda}n_\alpha \nu^{(\alpha)}$. Of course, $\nu-\sum_{\alpha \in\Lambda}n_\alpha \nu^{(\alpha)}\in \mathscr{R}_\om$. This proves \eqref{dirsum}.  Being a submodule of a free $\Z$-module, $\mathscr{R}_\om$ is also free (see for e.g. \cite{rotman2002advanced}). Let $(\xi^{(\beta)})_{\beta\in\Gamma}$, where $\Gamma$ is a countable set of indices, be a basis for $\mathscr{R}_\om$. Consider any matrix $A$ constructed so that each row is either an element of $(\nu^{(\alpha)})_{\alpha \in \Lambda}$ or of $(\xi^{(\beta)})_{\beta\in\Gamma}$, and such that each element of $(\nu^{(\alpha)})_{\alpha \in \Lambda}$ and of $(\xi^{(\beta)})_{\beta\in\Gamma}$ appears exactly once as a row of $A$. This matrix corresponds to an automorphism $\mathscr{A} \in \mathrm{Aut}(\T^\N)$ that conjugates $(\T^\N,\Phi_\om)$ to $(\T^\N,\Phi_{\tilde{\om}})$, where $\tilde{\om}:=A\om$. By construction, the non-vanishing components of $\tilde{\om}$ are rationally independent.
\medskip
		
\noindent $(\impliedby)$ Let $h:\T^\N\to \T^\N$ be a homeomorphism such that, for all $t$ in $\R$,  
\begin{equation}
h\circ\Phi^t_{\tilde{\om}}=\Phi^t_{\om}\circ h \ ,
\end{equation}
where $\tilde{\om}\in\R^\N$ is such that the non-vanishing components of $\tilde{\om}$ are rationally independent. Without loss of generality, we assume $h(0)=0$. Then, the claim follows from Lemma~\ref{lemmaB} and Remark \ref{remsupp}.
\end{proof}

We already pointed out in the introduction that if $\mathscr{M}_\om$ is finitely generated, and in particular if $\om$ has finite support, then $\mathscr{M}_\om$ is free. This follows from the structure theorem for finitely generated abelian groups without torsion. 

Finally, if the resonance module $\mathscr{R}_\omega$ has finite rank, Proposition~\ref{propprel} provides a conjugacy to a flow $(\mathbb{T}^\mathbb{N}, \Phi_{\tilde{\omega}})$ whose non-vanishing frequencies are rationally independent. By Theorem~\ref{B}, this ensures that $\mathscr{M}_\omega$ must be a free abelian group, even though it is no longer finitely generated.

\subsection{Solenoidal Orbits}\label{ratflow}
In the general case, the topological structure of the orbits of a Kronecker flow on the infinite torus can be much more complicated. Inspired by Sakbaev and Volovich \cite{SakVol}, here we construct a class of Kronecker flows such that the closure of any orbit induced by Kronecker flows in such class is a solenoid, namely a compact, connected subgroup of $\T^\N$ which is locally homeomorphic to the product of an interval and a Cantor set (see Proposition \ref{localcantor}). These orbits have the property that they are uniformly in time approximated by periodic ones, with the period diverging as the approximation becomes more accurate. In accordance with the previous section, the frequency modules associated with such Kronecker flows must fail to be free abelian groups. More precisely, we shall see in the next section that they are non-free subgroups of the additive rationals (see Lemma \ref{modsol} and Remark \ref{dualsol}).

\begin{notation}
	We denote by $\Sigma$ the set of sequences of natural numbers ${\textit{\textbf{a}}}=(a_j)_{j \in \N}$ with $a_1=1$, and $a_j>1$ for all $j>1$.
\end{notation}

\begin{defn}[Solenoid]\label{Solenoid}
Given $\textbf{a}\in \Sigma$, the solenoid $\sigma_{\textbf{a}}$ is the compact, connected subgroup of $ \ \T^\N$ defined as
		\begin{equation}\label{sol}
		\sigma_{\textbf{a}}:=\bigg\{ \theta=(\Theta_{j}+2\pi\Z)_{j \in\N} \in \T^\N \;  : \;  \Theta_j=a_{j+1}\Theta_{j+1} \; \mathrm{mod} \ 2\pi, \ \forall j \in \N \bigg\}.
		\end{equation}
\end{defn}

\begin{thm}\label{thmsol}
	Let $\textit{\textbf{a}}\in \Sigma$, and consider $\om(\textit{\textbf{a}})=(\om_j(\textit{\textbf{a}}))_{j \in \N}$ defined by
	\begin{equation*}
	\om_{j}(\textit{\textbf{a}})=\prod_{k=1}^{j}a_{k}^{-1} \quad \text{for all } j\in\N.
	\end{equation*}
	Then, the closure of any orbit induced by the Kronecker flow $(\T^\N,\Phi_{\om(\textit{\textbf{a}})})$ is homeomorphic to the solenoid $\sigma_{\textit{\textbf{a}}}$.
\end{thm}

\begin{proof}
	By Remark \ref{orb-Oom}, it is sufficient to prove the statement for the orbit through the origin. Note that for any time $t \in \R$, the point $(\om_j(\textit{\textbf{a}})t + 2\pi\Z)_{j \in \N}$ belongs to $\sigma_{\textit{\textbf{a}}}$. Since $\sigma_{\textit{\textbf{a}}}$ is closed, the closure of the orbit through the origin is necessarily contained in $\sigma_{\textit{\textbf{a}}}$.
	
	To show the reverse inclusion, let $\theta \in \sigma_{\textit{\textbf{a}}}$. We seek a sequence of real numbers $(t_k)_{k \in \N}$ such that $\lim_{k \to \infty}(\om_j(\textit{\textbf{a}})t_k + 2\pi\Z) = \theta_j + 2\pi\Z$ for all $j \in \N$. To this end, we first establish a one-to-one correspondence between $\sigma_{\textit{\textbf{a}}}$ and the set 
	\begin{equation}
	[0,1) \times \prod_{j=2}^{\infty}\{0,1,\dots,a_j-1\} .
	\end{equation}
	Indeed, $\theta \in \sigma_{\textit{\textbf{a}}}$ if and only if, for each $j\geq 2$, there exists $n_j\in \Z$ such that $a_{j}\check{\theta}_{j}-\check{\theta}_{j-1}=n_j$, where $\check{\theta}_j$ is the unique representative of $(2\pi)^{-1}\theta_j$ in $[0,1)$. Since $0 \leq \check{\theta}_j < 1$, it follows that
	\begin{equation}
	-1 < -\check{\theta}_{j-1} \leq a_j\check{\theta}_j - \check{\theta}_{j-1} = n_j < a_j ,
	\end{equation}
	which implies $n_j \in \{0,1,\dots,a_j-1\}$. Thus, each $\theta \in \sigma_{\textit{\textbf{a}}}$ uniquely determines a pair $(\tau, (n_j)_{j \geq 2})$ with $\tau = \check{\theta}_1 \in [0,1)$ and $n_j=a_j\check{\theta}_{j}-\check{\theta}_{j-1}$.
	
	Conversely, given $\tau \in [0,1)$ and $(n_j)_{j \geq 2}$ in the specified ranges, we associate the unique $\theta \in \sigma_{\textit{\textbf{a}}}$ such that $\check{\theta}_1=\tau$ and
	\begin{equation}\label{s2}
	\check{\theta}_j=\om_j(\textit{\textbf{a}})\bigg(\tau+\sum_{m=2}^{j}\dfrac{n_m}{\om_{m-1}(\textit{\textbf{a}})}\bigg) \quad \text{for all } j\geq 2.
	\end{equation}
	This association is well-defined since $\check{\theta}_j$ belongs to $[0,1)$, as shown by the following estimate:
	\begin{equation*}
	0 \leq \tau+\sum_{m=2}^{j}\dfrac{n_m}{\om_{m-1}(\textit{\textbf{a}})} = \tau+\sum_{m=2}^{j}n_m\prod_{k=1}^{m-1}a_k < 1+\sum_{m=2}^{j}(a_m-1)\prod_{k=1}^{m-1}a_k = \prod_{k=1}^{j}a_k = \dfrac{1}{\om_j(\textit{\textbf{a}})}.
	\end{equation*}
	It is straightforward to check that $\theta \in \sigma_{\textit{\textbf{a}}}$ and that these two mappings are inverses of each other.
	
	Finally, for a fixed $\theta \in \sigma_{\textit{\textbf{a}}}$, we define the sequence of times $(t_k)_{k \in \N}$ as
	\begin{equation}
	t_1 := 2\pi \check{\theta}_1, \quad t_k := 2\pi\bigg(\check{\theta}_1+\sum_{m=2}^{k}\dfrac{n_m}{\om_{m-1}(\textit{\textbf{a}})}\bigg) \quad \text{for all } k \geq 2.
	\end{equation}
	By construction, for any fixed $j \in \N$, we have $\lim_{k \to \infty}(\om_j(\textit{\textbf{a}})t_k + 2\pi\Z) = \theta_j + 2\pi\Z$ in the product topology. This confirms that $\mathcal{O}_{\om(\textbf{\textit{a}})} = \sigma_{\textit{\textbf{a}}}$.
\end{proof}

Let us give a simple description of the local topological structure of solenoids. 

\begin{prop}\label{localcantor}
	For all $\textit{\textbf{a}}\in\Sigma$, any point of $\sigma_{\textit{\textbf{a}}}$ admits a neighbourhood which is homeomorphic to the product of an interval and a Cantor set.
\end{prop}

\begin{proof}
	Let $\varphi=(\varphi_j)_{j \in \N} \in \sigma_{\textit{\textbf{a}}}$. Without loss of generality, we can always assume that $\varphi_1=\pi+2\pi\Z$. Consider the open neighbourhood $U$ of $\varphi$ in $\sigma_{\textit{\textbf{a}}}$ obtained by taking the intersection
	\begin{equation*}
	U := \bigg[ \Big( \T \setminus \{0 + 2\pi\Z\} \Big) \times \prod_{j=2}^{\infty} \T \bigg] \cap \sigma_{\textit{\textbf{a}}}.
	\end{equation*}
	
	We shall show that $U$ is homeomorphic to the product of the interval $(0,1)$ and the Cantor set $K := \prod_{j=2}^{\infty}\{0,1,\dots,a_j-1\}$. By the same argument used in the proof of Theorem \ref{thmsol}, $U$ is in a one-to-one correspondence with $(0,1) \times K$.
	
	Let us show that the map
	\begin{equation*}
	f: U \ni \theta \mapsto (\check{\theta}_1, (a_j\check{\theta}_j - \check{\theta}_{j-1})_{j=2}^\infty) \in (0,1) \times K
	\end{equation*}
	is a homeomorphism, where $(0,1)$ carries the standard topology, each $\{0,1,\dots,a_j-1\}$ is considered as discrete, and the infinite product is endowed with the product topology. Consider $\theta \in U$ and let $(\theta^k)_{k \in \N}$, $\theta^k \in U$ for all $k \in \N$, be a sequence converging to $\theta$. Then,
	\begin{gather*}
	\lim_{k \to \infty} (\theta^k_1 + 2\pi\Z) = \theta_1 + 2\pi\Z, \\
	\lim_{k \to \infty} (\theta^k_j + 2\pi\Z) = \theta_j + 2\pi\Z, \quad \forall j > 1.
	\end{gather*}
	Since $\theta \in U$, the first limit reads $\lim_{k \to \infty} \check{\theta}^k_1 = \check{\theta}_1$. The second limit is
	\begin{equation}\label{413}
	\lim_{k \to \infty} \bigg( \om_j \sum_{m=2}^j \frac{n^k_m}{\om_{m-1}} + \Z \bigg) = \om_j \sum_{m=2}^j \frac{n_m}{\om_{m-1}} + \Z, \quad \forall j \geq 2,
	\end{equation}
	where $n^k_m = a_m\check{\theta}^k_m - \check{\theta}^k_{m-1}$ and $n_m = a_m\check{\theta}_m - \check{\theta}_{m-1}$. But, for all $k \in \N$, we have
	\begin{equation*}
	\om_j \sum_{m=2}^j \frac{n^k_m}{\om_{m-1}} \leq 1 - \om_j < 1, \quad \text{for all } j \geq 2.
	\end{equation*}
	Thus, \eqref{413} reads
	\begin{equation*}
	\lim_{k \to \infty} \sum_{m=2}^j \frac{n^k_m}{\om_{m-1}} = \sum_{m=2}^j \frac{n_m}{\om_{m-1}}, \quad \forall j \geq 2,
	\end{equation*}
	which means that $\lim_{k \to \infty} n_j^k = n_j$ for all $j \geq 2$ because the $n_j^k$ take values in a discrete set. This proves that $f$ is continuous. The continuity of $f^{-1}$ follows from the fact that it is defined by the map \eqref{s2}, which is a composition of continuous functions.
	
	It remains to show that $K = \prod_{j=2}^{\infty}\{0,1,\dots,a_j-1\}$ is a Cantor set, namely, that it is compact, totally disconnected, and each point is an accumulation point. Compactness follows from Tychonoff's theorem. To see that $K$ is totally disconnected, which is a standard result for the product of discrete spaces, let $(n_j)_{j=2}^\infty \in K$ and consider $X \subseteq K$ such that $X$ contains $(n_j)_{j=2}^\infty$ and at least another distinct point $(n'_j)_{j=2}^\infty$. Let $k \in \N, k \geq 2$, be such that $n_k \neq n'_k$, and consider the open cylinder
	\begin{equation*}
	C := \{ (m_j)_{j=2}^\infty \in K : m_k = n_k \}.
	\end{equation*}
	Since each factor carries the discrete topology, the complement $\overline{C}$ of $C$ is also an open cylinder. Then, $X$ can be written as a union of disjoint non-empty open subsets $X = (C \cap X) \cup (\overline{C} \cap X)$. Finally, we observe that each point of $K$ is an accumulation point. Consider a list $(n_j)_{j=2}^\infty \in K$ and an open subset $X$ containing it. $X$ can be written as a union of open cylinders $\cup_\alpha C_\alpha$ such that $(n_j)_{j=2}^\infty$ belongs to some $C_\alpha$. Since $C_\alpha = \prod_{j=2}^\infty X_j$, where $X_j \neq \{0,1,\dots,a_j-1\}$ only for finitely many $j$, there always exists $(n'_j)_{j=2}^\infty \in C_\alpha$ such that $n_k \neq n'_k$ for some $k \geq 2$. This concludes the proof.
\end{proof}

\subsection{Orbits as Pontryagin Duals}\label{orbit-dual}
The first goal of this section is to show that, given a Kronecker flow on the infinite torus, the closure of any orbit is homeomorphic to the Pontryagin dual of the associated frequency module (Proposition \ref{mainprel}). Thereafter, given any countable family $(G_i)_{i\in I}$ of subgroups of $(\mathbb{Q},+)$, we construct a Kronecker flow such that its frequency module is isomorphic to $\bigoplus_{i \in I}G_i$. Then, by using Pontryagin's theory of locally compact abelian groups, together with Theorem \ref{thmsol} and Proposition \ref{mainprel}, we give the topological structure of the orbits. It turns out that the class of Kronecker flows so obtained is characterized by the fact that the closure of any orbit is a product of circles and solenoids (see Theorem \ref{productcirclesol}).

\paraga We begin by reviewing some fundamental definitions and results from Pontryagin's theory of locally compact abelian groups. For references, see the original paper by Lev Pontryagin \cite{Pon}, and \cite{morris} for a modern formulation of the theory. 

\begin{notation}
	In the following, the symbol $\cong$ denotes an isomorphism of topological groups. Consistently with the definition of $ \; \mathbb{T}$ given in \ref{notations}, the group operation on the one dimensional torus is always expressed in additive notation and denoted by $+$.
\end{notation}

\begin{defn}[Character]\label{cha}
	Let $G$ be a locally compact abelian group. A character on $G$ is a continuous homomorphism of groups $\chi:G\to \T$.
\end{defn}
\begin{defn}[Pontryagin Dual]\label{pd}
	The Pontryagin dual $G^*$ of a locally compact abelian group $G$ is the group of all characters on $G$ endowed with the compact-open topology. The group operation $*$ is defined by
	\begin{equation*}
	(\chi*\psi)(g):= \chi(g)+\psi(g), \quad \forall \chi,\psi \in G^*, \quad g \in G.
	\end{equation*}
\end{defn}

\begin{thm}
	The Pontryagin dual of a locally compact abelian group is indeed a locally compact abelian group.
\end{thm}

\begin{thm}\label{discrcomp}
	The Pontryagin dual of a discrete abelian group is compact.
\end{thm}

\begin{ex}\label{extor}
The Pontryagin dual of $\Z$ is the torus $\T$, with the dual coupling given by
	\begin{equation}
	\theta:\Z \ni n \mapsto n\Theta + 2\pi\Z \in \T.
	\end{equation}
\end{ex}

\begin{thm}\label{dirsumdual}
	Let $J$ be a countable set of indices and consider $(G_j)_{j \in J}$ a family of discrete abelian groups. Then, the direct sum $\bigoplus_{j \in J}G_{j}$ is a discrete abelian group when considered with the product topology, and its Pontryagin dual is the (compact) abelian group $\prod_{j \in J}(G_j)^*$ endowed with the product topology.
\end{thm}

\begin{ex}\label{dualZN}
	Let $J$ be a countable set of indices. By Example \ref{extor} and Theorem \ref{dirsumdual}, the Pontryagin dual of $\bigoplus_{j\in J}\Z$ is the torus $\prod_{j\in J}\T$, with  dual coupling is given by
	\begin{equation}
	\theta:\bigoplus_{j\in J}\Z \ni \nu \mapsto \sum_{j\in J}\nu_j\Theta_j + 2\pi\Z \in \T.
	\end{equation}
\end{ex} 

\begin{defn}[Evaluation Character/Map]
	Given an element $g$ of a locally compact abelian group $G$, the associated evaluation character $\mathrm{ev}_{G}(g)$ is the character on $G^*$ defined by
	\begin{equation}
	\mathrm{ev}_{G}(g)(\chi):=\chi(g), \quad \forall \chi \in G^*.
	\end{equation}
	The map $\mathrm{ev}_{G} : G \to G^{**}$ is called eval\-ua\-tion map.
\end{defn}

\begin{thm}[Pontryagin Duality]\label{pontryaginthm}
	For any locally compact abelian group, the evaluation map is an isomorphism of topological groups.
\end{thm}

\begin{defn}[Annihilator]
	Let $G$ be a locally compact abelian group and $H$ a closed subgroup. The annihilator of $H$ is the (locally compact) subgroup of $G^*$ defined by
	\begin{equation}\label{defann}
	Ann(H):=\{ \chi \in G^* \; : \; \chi(g)=0, \; \forall g \in H \}.
	\end{equation}
\end{defn}

\begin{thm}\label{annann}
	Let $G$ be a locally compact abelian group and $H$ a closed subgroup. Then, $Ann(Ann(H))\cong H$.
\end{thm}

\begin{prop}
	Let $G$ be a locally compact abelian group and $H$ a closed subgroup. Then, $G/H$ is a locally compact abelian group.
\end{prop}

\begin{thm}\label{dualquoz}
	Let $G$ be a locally compact abelian group and $H$ a closed subgroup. Then, $(G/H)^*\cong Ann(H)$.
\end{thm}

\paraga With the above classical results in the theory of locally compact abelian groups, we can now proceed with exploring their implications in our context.

\begin{lemma}\label{OM}
	Let $\om:=(\om_j)_{j\in \N}$, be a list of real numbers and $\mathcal{O}_\om$ the closure of the orbit through the origin induced by the Kronecker flow $(\T^\N,\Phi_\om)$. Then, the annihilator of $\mathcal{O}_\om$ is the resonance module 
	\begin{equation}
	\mathscr{R}_\om:=\bigg\{\nu \in \Z^{(\N)} \; : \; \sum_{j\in \N} \nu_j\om_j=0 \bigg\}.
	\end{equation}
	\begin{proof}
		Let $\nu \in Ann(\mathscr{O}_\om)$. Then, $t\sum_{j \in \N}\om_j\nu_j\in 2\pi\Z$ for all $t \in \R$. This implies that $\nu \in \mathscr{R}_\om$. Now, assume that $\nu \in \mathscr{R}_\om$. For any $\theta \in \mathcal{O}_\om$ there exists a sequence $\{t_k\}_{k\in \N}$ such that, for all $j\in \N$, $\Theta_j+2\pi\Z=\lim_{k \to \infty}t_k\om_j+2\pi\Z$. Then, 
		\begin{equation*}
		\sum_{j\in \N}\Theta_j\nu_j+2\pi\Z=\sum_{j\in \N}\lim_{k \to \infty}t_k\om_j\nu_j+2\pi\Z=0+2\pi\Z=0,	
		\end{equation*}
		which ends the proof.
	\end{proof}
\end{lemma}

\begin{prop}\label{D0}
	Given a Kronecker flow $(\T^\N,\Phi_\om)$ on the infinite torus, the closure of the orbit through the origin is isomorphic as a topological group to $\mathscr{M}_\om^*$, where the frequency module $\mathscr{M}_\om$ is endowed with the discrete topology.
	\begin{proof}
		First observe that, by the first isomorphism theorem (see Remark \ref{firstisothm}), $\mathscr{M}_\om$ is isomorphic to $\Z^{(\N)}/\mathscr{R}_\om$. Then,
		\begin{equation}
		\mathscr{M}_\om^*\cong\bigg(\Z^{(\N)}/\mathscr{R}_\om\bigg)^* \cong Ann(\mathscr{R}_\om) \cong Ann(Ann(\mathcal{O}_\om))\cong \mathcal{O}_\om.
		\end{equation}
The last isomorphisms follow from Theorem \ref{dualquoz}, Lemma \ref{OM}, and Theorem~\ref{annann}, respectively.
	\end{proof}
\end{prop}

\begin{prop}\label{mainprel}
	Given a Kronecker flow $\Phi_\om$ on the infinite torus, the closure of any orbit is homeomorphic to $\mathscr{M}_\om^*$.
	\begin{proof}
 The claim follows by Remark \ref{orb-Oom} and Proposition \ref{D0}.
	\end{proof}
\end{prop}

\paraga As a consequence of a characterization of all the subgroups of the additive rationals (see \cite{beau}), we have the following lemma, which is proved in Appendix \ref{modsolA}.
\begin{lemma}\label{modsol}
	Any subgroup of $(\Q,+)$ is either free or isomorphic to
	\begin{equation}
	Q(\textnormal{\textit{\textbf{a}}}):=\bigg\{ \dfrac{m}{\prod_{j=1}^Na_j} \; : \; m\in\Z \ \text{and} \ N \in \N\bigg\},
	\end{equation}
	for some sequence $\textnormal{\textit{\textbf{a}}}\in \Sigma$.
\end{lemma}

\begin{rmk}\label{dualsol}
	Given $\textbf{a}\in\Sigma$, consider the non-free subgroup $Q(\textbf{a})$ of the additive rational. Then, there exists a Kronecker flow whose frequency module is isomorphic to the given subgroup. Indeed, if $\om=\om(\textbf{a})$, where
	\begin{equation}\label{omegaa}
	\om(\textbf{a})=(\om_j(\textbf{a}))_{j \in \N}, \quad \text{with} \quad \om_{j}(\textbf{a})=\prod_{k=1}^{j}a_{k}^{-1} \quad \text{for all} \ j\in\N,
	\end{equation}
	then $\mathscr{M}_\om=Q(\textbf{a})$.
	
	A priori, there may be other possible choices of the frequency vector $\om$ that give the same frequency module. Nevertheless, Proposition \ref{mainprel} ensures that the topological structure of the closure of the orbits of any Kronecker flow is uniquely determined by the frequency module. Finally, Theorem \ref{thmsol} implies that the Pontryagin dual of $Q(\textbf{a})$ is the solenoid $\sigma_{\textbf{a}}$.
\end{rmk}

\begin{rmk}\label{rmksakvol}
	 The key idea of the construction of Sakbaev and Volovich \cite{SakVol}, consists in considering orbits that are (uniformly in time) approximated by periodic ones, whose period diverges as the approximation becomes more accurate. In particular, they provide two examples, namely, $\om_{j}=1/j!, \ j \in \N$, and $\om'_j=\textbf{p}_{2j}/\textbf{p}_{2j-1}, \ j\in \N$, where $(\textbf{p}_j)_{j\in \N}$ is the list of all prime numbers in ascending order. Note that the closures of the orbits induced by $\Phi_\om$ and $\Phi_{\om'}$ are homeomorphic to the solenoids $\sigma_{\textbf{a}}$, with $\textbf{a}=(1,2,3,4,...)$, and $\sigma_{\textbf{a}'}$, with $\textbf{a}'=(1,\textbf{p}_1,\textbf{p}_3,\textbf{p}_5,...)$, respectively. Indeed, with the same notations of Lemma \ref{modsol}, it is easy to verify that
	\begin{equation}
		\mathscr{M}_\om=Q(\textbf{a}), \quad \mathscr{M}_{\om'}=Q(\textbf{a}').
	\end{equation}
	The conclusion follows by Proposition \ref{mainprel} and Remark \ref{dualsol}.
\end{rmk}

To conclude this section, we show that, given a countable family $(G_i)_{i\in I}$ of subgroups of $(\mathbb{Q},+)$, we are able to construct a Kronecker flow such that its frequency module is isomorphic to $\bigoplus_{i \in I}G_i$ and we give the general structure of the orbits. First, we need the following
\begin{prop}\label{unique}
	Let $G$ be a subgroup of $(\R,+)$. Assume that there exist a set of indices $I$ and a family of rationally independent numbers $(\alpha_i)_{i \in I}$ such that for any $g\in G$ there is a finite subset $J\subseteq I$ and $(r_i)_{i \in J}$, $r_i \in \Q$ for all $i \in J$, satisfying $g=\sum_{i \in J}r_i\alpha_i$. Then, $G$ is isomorphic to $\bigoplus_{i \in I}R_i$, where, for any $i \in I$, $R_i$ is the following subgroup of $(\Q,+)$:
	\begin{equation}
	R_i:=\alpha_i^{-1}\{g \in G \; : \; g=\alpha_i r, \ \text{for some} \ r\in\Q\}.
	\end{equation}
	\begin{proof}
		Since $(\alpha_i)_{i \in I}$ are rationally independent, then for any finite subset $J\subseteq I$, the equation $\sum_{i \in J}\alpha_ir_i=0$  has a unique solution in $\bigoplus_{i \in J}\Q$, that is, $r_i=0$ for all $i \in J$.
	\end{proof}
\end{prop}

\begin{thmC}\label{productcirclesol}
	Let $I$ be a countable set of indices and $(G_i)_{i \in I}$ a family of subgroups of $(\mathbb{Q},+)$. Then, there exists a Kronecker flow whose frequency module is isomorphic to $\bigoplus_{i \in I} G_i$. Moreover, the closure of any orbit is homeomorphic to a product $\prod_{i \in I} X_i$, where each factor $X_i$ is a circle if $G_i$ is free, and a solenoid otherwise.
\end{thmC}

\begin{proof}
By taking an arbitrary enumeration of $I$, we can think at $I$ as a subset of $\N$. By Lemma \ref{modsol}, $\bigoplus_{i \in I}G \cong \bigoplus_{k \in K}\Z \ \oplus \  \bigoplus_{n \in J}Q({\textit{\textbf{a}}^n})$, for some $J,K\subseteq I$, $J \cap K = \emptyset$, and some family $(\textit{\textbf{a}}^{n})_{n \in J}$ of sequences of natural numbers such that, for any $n \in J$, $\textit{\textbf{a}}^{n}\in \Sigma$. Let $(\textbf{p}_n)_{n \in \N}$ be the list of all prime numbers in ascending order. Then, define $\om \in \R^\N$ as
\begin{equation}
\om_{j}:=\dfrac{\sqrt{\textbf{p}_n}}{\prod_{k=1}^{N}a^n_k} \ \ \ \text{if} \ j=\textbf{p}_n^{N}, \ \text{for some} \ n \in J, \ \text{and} \ N\in \N,
\end{equation}
\noindent and we complete the construction of $\om$ by associating powers of $\pi$, i.e., $\pi^k$, for all $k \in K$, and zeros to all the remaining components. By Proposition \ref{unique}, an easy calculation gives
\begin{equation}
\mathscr{M}_\om\cong\bigoplus_{k \in K}\Z \ \oplus \ \bigoplus_{n \in J}\bigg\{ \dfrac{m}{\prod_{k=1}^Na^{n}_k} \; : \; m\in\Z \ \text{and} \ N \in \N \bigg\}.
\end{equation}
Furthermore, by Proposition \ref{mainprel}, Proposition \ref{unique} and Remark \ref{dualsol}, the closure of any orbit induced by $(\T^\N,\Phi_\om)$ is homeomorphic to the direct product of circles and solenoids
\begin{equation}
\prod_{k \in K}\T \ \times \ \prod_{n \in J}\sigma_{\textit{\textbf{a}}^n}.
\end{equation}
\end{proof}

\subsection{Solenoidal Solutions of the Benjamin-Ono Equation}\label{ssbo}
As an application of the results obtained in the previous section, we show that the Benjamin-Ono equation for a $2\pi$-periodic in space, real valued function
\begin{equation}\label{BOeq}
u_t=\mathrm{H}(u_{xx})-2uu_x, \ \ \ x \in \T, \ t \in \R,
\end{equation}
where $\mathrm{H}( \ \cdot \ )$ denotes the Hilbert transform 
\begin{equation}\label{htransform}
\mathrm{H}(u)(x,t)=-\frac{\im}{2\pi}\sum_{j \in \Z}\sgn(j)\int_{0}^{2\pi}u(x',t)e^{\im j(x- x')}dx',
\end{equation}
admits a family of (non-typical) almost-periodic solutions that are dense on invariant subsets of $L^2(\T)$ which are homeomorphic to the product of a circle and a solenoid.

\begin{notation}
In the following, $L^2_0(\T)$ denotes the space of functions $u \in L^2(\T)$ with $\int_0^{2\pi}u(x)dx=0$.
\end{notation}
The following result is proved in \cite{BO} by Kappeler and Gerard.
\begin{thm}[\cite{BO}]\label{KG}
There exists a global bi-analytic diffeomorphism
\begin{equation}
\Psi:L_0^2(\T)\ni u \mapsto z \in h^{1/2}:=\bigg\{z\in \C^\N \; : \; \|z\|^2_{1/2}:=\sum_{j\in \N} j|z_j|^2<\infty\bigg\},
\end{equation}
\noindent called \textquotedblleft Birkhoff map", such that, for any initial data $\underline{u}\in L^2_0(\T)$, the solution of the Benjamin-Ono equation \eqref{BOeq} --- that is supported on the invariant torus $\Psi^{-1}(\mathcal{T}_{\gamma})$
\begin{equation}\label{tc}
\mathcal{T}_{\gamma}:=\{z \in h^{1/2} \; : \; |z_j|^2=\gamma_j:=|\Psi_j(\underline{u})|^2, \ \forall j\in \N  \}
\end{equation}
whose dimension is equal to the number of non-zero \textquotedblleft Birkhoff coodinates" associated with the initial data --- is given by
\begin{equation}\label{BOflow}
\Psi_j(u(t))=\Psi_j(\underline{u})e^{\im \om_j(\gamma)t}, \ \ \ \forall j \in \N,
\end{equation}
with frequencies
\begin{equation}\label{BOfreq}
\om_j(\gamma):=j^2-2\sum_{k\in \N} \min(j,k)\gamma_k, \quad \forall j \in \N.
\end{equation}
\end{thm}

\begin{rmk}\label{conjBO}
	It is easy to verify that the tori defined in \eqref{tc}, when considered with the subspace topology induced by the $\|\cdot\|_{1/2}$-norm, are homeomorphic to the, possibly infinite, torus $\prod_{j\in\text{supp}(\gamma)}\T$. Namely, the map
	\begin{equation}
	\Xi:\prod_{j\in\text{supp}(\gamma)}\T \ni \theta \mapsto \Xi(\gamma):=(\sqrt{\gamma_j}e^{\im \Theta_j})_{j\in \N}\in \mathcal{T}_{\gamma}\subset h^{1/2},
	\end{equation}
	where $\Theta_j:=0$ for all $j\notin \text{supp}(\gamma)$,
	is a topological embedding. Therefore, since the Birkhoff map $\Psi$ is a homeomorphism, the dynamics induced by the Benjamin-Ono equation on any invariant set $\Psi^{-1}(\mathcal{T}_\gamma)$ is topologically conjugate by
	\begin{equation}\label{Hconj}
	\mathscr{H}:=\Psi^{-1}\circ \Xi \ 
	\end{equation}
	to the Kronecker flow
	\begin{equation}
	\Phi_{\om(\gamma)}:\mathbb{R}\times \prod_{j\in \text{supp}(\gamma)} \T \to \prod_{j\in \text{supp}(\gamma)} \T \ , \quad (t,\theta) \mapsto \Phi^t_{\om(\gamma)}(\theta):=(\Theta_j +\om_j(\gamma) t +2\pi\Z)_{j\in \text{supp}(\gamma)}.
	\end{equation}
	with $\om(\gamma)$ given by \eqref{BOfreq}.
\end{rmk}

The goal of this section is to give a proof of the following theorem.
\begin{thmBO}\label{BO}
Consider the Benjamin-Ono equation for a real-valued function $u(t,x)$ that is $2\pi$-periodic in space:
\begin{equation}
u_t=\mathrm{H}(u_{xx})-2uu_x, \quad x \in \mathbb{T}, \; t \in \mathbb{R},
\end{equation}
where $\mathrm{H}(\cdot)$ denotes the Hilbert transform defined in \eqref{htransform}. Let
\begin{equation}
\Psi:L_0^2(\mathbb{T})\ni u \mapsto z \in h^{1/2}
\end{equation}
be its Birkhoff map \cite{BO}. Given $z\in h^{1/2}$, if there exists $\beta\in \mathbb{R}\setminus \mathbb{Q}$ such that $|z_j|^2\in \beta \mathbb{Q}\setminus\{0\}$ for all $j \in \mathbb{N}$, then the closure of the orbit with initial data $u = \Psi^{-1}(z)$ under the Benjamin-Ono flow is homeomorphic to the product of a circle and a solenoid.
\end{thmBO}

In the above theorem, the closure is taken with respect to the subspace topology induced by the $L^2$-norm. Then, by Remark \ref{conjBO}, it is sufficient to study the orbits of a Kronecker flow with frequencies defined as in \eqref{BOfreq} on the infinite torus $\T^\N$ endowed with the product topology.

\begin{prop}\label{split}
	Let $\alpha$ be an irrational number and $\om=(\om_j)_{j \in \N} \in \R^\N$ a list defined by 
	\begin{equation}
	\om_j=j^2+\alpha\sum_{k\in \N} \min(j,k)s_k 
	\end{equation}
	where $(s_j)_{j\in \N}$ is a list of rational numbers such that $\sum_{j \in \N} js_j$ is finite. Then, the frequency module $\mathscr{M}_\om$ is isomorphic to $\Z \oplus R$, where $R$ is the subgroup of $(\Q,+)$ defined by
	\begin{equation}\label{S}
	R:=\bigg\{\sum_{j\in \N} \nu_j\sum_{k\in \N} \min(j,k)s_k \; : \; \nu\in\Z^{(\N)}   \bigg\}.
	\end{equation}
\end{prop}
\begin{proof}
	The claim follows from Proposition \ref{unique} and from the fact that
	\begin{equation*}
	\bigg\{\sum_{j\in \N} \nu_j j^2 : \nu \in \Z^{(\N)} \bigg\} = \Z,
	\end{equation*}
	which holds since the set of squares $\{j^2 : j \in \N\}$ generates the entire additive group $(\Z, +)$. This is immediately verified by considering the map $n \mapsto n e_1$, where $e_1$ is the first vector of the canonical basis of $\Z^{(\N)}$. Since the weighted sum associated with $e_1$ is $1^2 = 1$, and $1$ is the generator of $\Z$, the image of the mapping $\nu \mapsto \sum \nu_j j^2$ must coincide with $\Z$. Formally, the existence of $e_1$ such that $\sum (e_1)_j j^2 = 1$ ensures that the generator of the target group is reached, making any further contributions from $j > 1$ redundant for the purpose of surjectivity.
\end{proof}

\begin{prop}\label{TS*}
	Given a Kronecker flow $\Phi_\om$ on the infinite torus, with
	\begin{equation}
	\om_j:=j^2+\alpha\sum_{k\in \N} \min(j,k)s_k, \ \ \ \alpha \in \R\setminus \Q,
	\end{equation}
	where  $(s_j)_{j\in \N}$ is a list of rational numbers such that the sum $\ \sum_{j\in \N} js_j$ is finite, the closure of any orbit is homeomorphic to $\T \times R^*$, where $R$ is defined in \eqref{S}.
	\begin{proof}
		By Proposition \ref{mainprel} the closure of any orbit is homeomorphic to the Pontryagin dual of $\Z \oplus R$, which, by Theorem \ref{dirsumdual}, Proposition \ref{split}, and Example \ref{extor}, turns out to be equal to $\T\times R^*$.
	\end{proof} 
\end{prop}

Now we prove that, if the support of $(s_j)_{j\in \N}$ is not finite, and each $s_j$ is non-negative, the subgroup $R\subset (\Q,+)$, defined in \eqref{S}, is not free. The strategy is based on the following lemma, which is proved in Appendix \ref{notfreeA}.
\begin{lemma}\label{notfree}
	Let $G$ be a subgroup of $(\Q,+)$. If there exists a sequence $(g_n)_{n\in\N}\subset G$, $g_n\neq 0$ for all $n \in \N$, converging to $0$ with respect to the subspace topology induced by the standard one on $\R$, then, $G$ is not free.
\end{lemma}

\begin{prop}\label{snotfree}
	With the same notations as in Proposition \ref{TS*}, if moreover the support of $(s_j)_{j\in\N}$ is not finite, and $s_j\geq0$ $\forall j\in\N$, the subgroup $R\subset (\Q,+)$ defined in \eqref{S} is not free.
	\begin{proof}
		Consider the sequence $(\nu^n)_{n\in\N} \subset \Z^{(\N)}$ defined by $\nu^n=e_{n+1}-e_n$. Then,
		\begin{align*}
		g_n:&=\sum_{j=1}^\infty \nu^n_j\sum_{k=1}^\infty \min(j,k)s_k=\sum_{k=1}^{\infty}\min(n+1,k)s_k-\sum_{k=1}^\infty \min(n,k)s_k\\
		&=\sum_{k=1}^{n+1}ks_k+(n+1)\sum_{k=n+2}^\infty s_k-\sum_{k=1}^nks_k-n\sum_{k=n+1}^\infty s_k\\
		&=(n+1)s_{n+1}+n\bigg(\sum_{k=n+2}^\infty s_k-\sum_{k=n+1}^\infty s_k\bigg)+\sum_{k=n+2}^\infty s_k\\
		&=(n+1)s_{n+1}-ns_{n+1}+\sum_{k=n+2}^\infty s_k\\
		&=\sum_{k>n}s_k.
		\end{align*}
		By assumption, the support of $(s_j)_{j\in \N}$ is not finite, and each $s_j$ is non-negative. Then, $g_n\neq 0$ for all $n \in \N$, and the sequence $(g_n)_{n\in\N}\subset R$ converges to $0$. The claim follows by Lemma \ref{notfree}.
	\end{proof}
\end{prop}

\begin{prop}\label{lastbo}
	Let  $\Phi_\om$ be a Kronecker flow on the infinite torus where $\om$ satisfies the same assumptions as in Proposition \ref{TS*}. If we also assume that the support of $(s_j)_{j\in \N}$ is not finite, and each $s_j$ is non-negative, then the closure of any orbit is homeomorphic to the product of a circle and a solenoid.
	\begin{proof}
		The claim follows directly from Propositions \ref{TS*} and \ref{snotfree}, Lemma \ref{modsol} and Remark~\ref{dualsol}.
	\end{proof}
\end{prop}

We are now in a position to prove Theorem \ref{BO}.

	\begin{proof}[Proof of Theorem \ref{BO}]
		Let $z\in h^{1/2}$ be such that $|z_j|^2\in \beta \Q\setminus\{0\}$, for all $j \in \N$, for a certain $\beta\in \R\setminus \Q$. Let $(\gamma_j)_{j \in \N}$, $\gamma_j:=|z_j|^2$ for all $j \in \N$, and define $\mathcal{T}_\gamma$ as $$\mathcal{T}_\gamma:=\{w \in h^{1/2} \; : \; |w_j|^2=\gamma_j, \; \forall j\in \N \}.$$	As already observed in Remark \ref{conjBO}, the map $\Xi$ defined by
		\begin{equation*}
		\Xi:\prod_{j\in \N}\T \ni \theta \mapsto \Xi(\gamma):=(\sqrt{\gamma_j}e^{\im \Theta_j})_{j\in \N}\in \mathcal{T}_{\gamma}\subset h^{1/2}
		\end{equation*}
		is a topological embedding. Therefore, the dynamics induced by the Benjamin-Ono equation on the invariant torus $\Psi^{-1}(\mathcal{T}_\gamma)$
		is topologically conjugate by
		\begin{equation*}
		\mathscr{H}:=\Psi^{-1}\circ \Xi \ 
		\end{equation*}
		to a Kronecker flow $\Phi_{\om(\gamma)}$ with frequencies $\om(\gamma)=(\om_j(\gamma))_{j\in \N}$, where
		\begin{equation*} 
		\om_j(\gamma):=j^2-2\sum_{k\in \N} \min(j,k)\gamma_k \quad \text{for all} j \in \N.
		\end{equation*}
Therefore, the claim follows from Proposition \ref{lastbo}.
	\end{proof}

\subsection{A Classification Problem}\label{stocazzo}
Here we address a classification problem for Kronecker flows on the infinite torus. The questions is: given two Kronecker flows, are we able to determine whether the closures of their orbits are homeomorphic? To address this problem, we make use of Pontryagin's Duality Theorem \ref{pontryaginthm}, along with a result by Scheinberg \cite{Scheinberg_1974} concerning the equivalence of the notions of homeomorphism and isomorphism for connected, locally compact abelian groups.

\begin{thmD}\label{E}
	Given $\om,\tilde{\om} \in \R^\N$, let $(\mathbb{T}^\mathbb{N}, \Phi_\omega)$ and $(\mathbb{T}^\mathbb{N}, \Phi_{\tilde{\omega}})$ be two Kronecker flows on the infinite torus. The closure of any orbit of $\Phi_\omega$ is homeomorphic to the closure of any orbit of $\Phi_{\tilde{\omega}}$ if and only if the associated frequency modules $\mathscr{M}_\omega$ and $\mathscr{M}_{\tilde{\omega}}$ are isomorphic as abelian groups.
\end{thmD}
	\begin{proof}
		By Remark \ref{orb-Oom}, it is sufficient to prove the statement in the case of orbits through the origin. In \cite{Scheinberg_1974} it is proved that two connected, locally compact abelian groups are homeomorphic if and only if they are isomorphic as topological groups. Then, the statement of Theorem \ref{E} is equivalent to the following: given two Kronecker flows on the infinite torus, the closures of the orbits through the origin are isomorphic as topological groups if and only if their frequency modules are isomorphic. Let us prove it. 
		
		Let $\tilde{\om},\om \in \R^\N$ and consider the Kronecker flows $(\T^\N,\Phi_\om)$, $(\T^\N,\Phi_{\tilde{\om}})$. If $\mathscr{M}_\om \cong \mathscr{M}_{\tilde{\om}}$, then, by Proposition \ref{D0},
		\begin{equation}
		\mathcal{O}_\om\cong (\mathscr{M}_\om)^* \cong (\mathscr{M}_{\tilde{\om}})^*\cong \mathcal{O}_{\tilde{\om}}.
		\end{equation}
		Now, assume that $\mathcal{O}_\om \cong \mathcal{O}_{\tilde{\om}}$. Then, by Theorem \ref{pontryaginthm} and Proposition \ref{D0},
		\begin{equation}
		\mathscr{M}_\om \cong (\mathscr{M}_\om)^{**} \cong (\mathcal{O}_\om)^*\cong (\mathcal{O}_{\tilde{\om}})^* \cong (\mathscr{M}_{\tilde{\om}})^{**} \cong \mathscr{M}_{\tilde{\om}}.
		\end{equation}
	\end{proof}

We now describe how to associate to each countable abelian group $G$ without torsion a Kronecker flow $(\T^\N,\Phi_\om)$ in such a way that $G$ is isomorphic to $\mathscr{M}_\om$. Any countable abelian group $G$ without torsion embeds canonically into a countable vector space $\mathscr{V}$ over $\Q$. This procedure, briefly explained in Appendix \ref{AAB}, is known in commutative algebra as "localization of an abelian group". Assuming Zorn's lemma, there exists a Hamel (algebraic) basis for $\mathscr{V}$. Thus, there exists a countable set of indices $J$ such that $\mathscr{V}$ is isomorphic to $\Q^{(J)}$. In particular, any element $x \in \mathscr{V}$ is uniquely represented by a family of rational numbers $(x_j)_{j \in J}$ such that $x_j\neq0$ only for finitely many $j \in J$. Now, observe that the map
\begin{equation}
	\textbf{h}:\mathscr{V}\ni(x_j)_{j \in J}\mapsto \sum_{j \in J}\alpha_jx_j\in (\R,+),
\end{equation} 
where $(\alpha_j)_{j \in J}$ is a family of rationally independent real numbers,
is an injective group homomorphism, and so is the composition $\textbf{H}:=\textbf{h}\circ \imath$, where $\imath:G \to \mathscr{V}$ is the injective homomorphism $\imath:G\to \mathscr{V}$ that embeds $G$ into $\mathscr{V}$. Then, by the first isomorphism theorem, $G$ is isomorphic to the image of $\textbf{H}$ in $(\R,+)$. Finally, we define the list of real numbers $\om \in \R^\N$ by taking an arbitrary enumeration of the image of $\textbf{H}$. By construction, $\mathscr{M}_\om\cong G$. 
 
 Then, this argument, together with Theorem \ref{E}, implies that completely classifying Kronecker flows on $\T^\N$ up to isomorphism of their orbits through the origin, is a highly complex issue, as it is equivalent to the classification of countable abelian groups without torsion. To date, a complete classification of such groups is known only for the rank $1$ case, i.e., subgroups of $(\Q,+)$ (for references, see \cite{baer}, \cite{beau} and \cite{Fuchs2015})\footnote{For the general theory of infinite abelian groups we refer to Kaplansky's book \cite{kapl}.},  which corresponds to the class of Kronecker flows (constructed in Section \ref{orbit-dual}) whose orbits are dense in a product of circles and solenoids. 
 \bigskip
 
 \noindent \textbf{Acknowledgements.} The Author acknowledges the support of the project “Stable and unstable phenomena in propagation of Waves in
 dispersive media” of INdAM-GNAMPA.
 
 The Author is deeply grateful to Jean-Pierre Marco for his invaluable guidance and for the many insightful discussions during various visits to Sorbonne Université; his ideas were a primary source of inspiration for this work. Special thanks go to Claudio Procesi for his helpful suggestions and for sharing his profound expertise in algebra, which proved to be fundamental to the development of this paper. Finally, the Author would like to thank Michela Procesi and Jessica Elisa Massetti for their constant support and for the numerous stimulating discussions that greatly contributed to this work.

\appendix
\section{Product Topology}\label{AAA}
\begin{defn}[Product Topology] Let $(X_k,\tau_k)_{k \in \mathbb{N}}$ be a family of topological spaces and $X:=\prod_{k\in \mathbb{N}}X_k$. The product topology $\tau_{prod}$ on $X$ is the coarsest topology that makes all the projections $q_k: X \to X_k$ continuous. It is then clear that a basis for the product topology consists of all sets obtained as finite intersections of sets of the form $q_k^{-1}(A_k)$, where $A_k$ is open in $X_k$. An open cylinder is a subset of $X$ of the form $\prod_{k\in \mathbb{N}}A_k$, where, for all $k \in \mathbb{N}$, $A_k \in \tau_k$, and $A_k \neq X_k$ only for a finite number of $k \in \mathbb{N}$. Therefore, the product topology on $X$ is generated by the open cylinders, i.e., $U \in \tau_{prod}$ if and only if $U$ is an arbitrary union of open cylinders. 
\end{defn} 

\begin{rmk}
	 Let $(X_k,\tau_k)_{k \in \mathbb{N}}$ be a family of topological spaces and $X:=\prod_{k\in \mathbb{N}}X_k$. If, for any $k \in \N$, $(X_k,\tau_k)$ is Hausdorff, then, $(X,\tau_{prod})$ is Hausdorff.
\end{rmk}

\begin{prop}[Convergence]\label{conv} A sequence $(x^n)_{n \in \mathbb{N}}$ in $X$ converges if and only if, for every $k \in \mathbb{N}$, $(x^n_k)_{n \in \mathbb{N}}$ converges in $X_k$. 
	\begin{proof} Since all the projections $q_k: X \to X_k$ are continuous, if $(x^n)_{n \in \mathbb{N}}$ converges to $x$ in $X$, then $(x_k^n)_{n \in \mathbb{N}}$ converges to $x_k$ in $X_k$ for each $k \in \mathbb{N}$. Conversely, let $x \in X$ and consider a sequence $(x^n)_{n \in \mathbb{N}}$ such that, $\forall k \in \mathbb{N}$, the sequence $(x^n_k)_{n \in \mathbb{N}}$ converges to $x_k$ in $X_k$. Since any open set is union of open cylinders, it is sufficient to show that, for any cyclinder containing $x$, $\exists N \in \mathbb{N}$ such that $x^n \in C$ $\forall n > N$. So, let $C=\prod_{k \in \mathbb{N}}A_k$ be an open cylinder such that, for all $k \in \mathbb{N}$, $x_k \in A_k$. Of course, there exists $K \in \mathbb{N}$ such that $A_k=X_k$ for all $k>K$. Now, by hypothesis, for any $k \in \mathbb{N}$ there is $N(A_k)$ such that $x^n_k \in A_k$ for all $n > N(A_k)$. But then, it is sufficient to take $N:=\sup_{k\leq K}N(A_k)$.
	\end{proof}
\end{prop}

\begin{rmk}  
	Let $(X,d)$ be a metric space. The topology generated by the metric $d$ coincides with the topology generated by $d':=d/(1+d)$. In particular, if the goal is to study only the topological structure of a metric space with metric $d$, one can always consider the metric $d'$, which takes values in $[0,1)$.  
	\end{rmk}  
Let $(X_k,d_k)_{k \in \mathbb{N}}$ be a family of metric spaces such that, for every $k \in \mathbb{N}$,  
$d_k:X_k\times X_k \to [0,1) \ .$ For every $k \in \mathbb{N}$, let $\tau_k$ be the topology on $X_k$ generated by the metric $d_k$. Let $(\rho_k)_{k \in \mathbb{N}}$ be a sequence of real numbers such that, $\forall k \in \mathbb{N}$, $\rho_k>0$, and $\sum_{k \in \mathbb{N}} \rho_k < \infty$. We define the metric on $X:=\prod_{k\in \mathbb{N}}X_k$  
$$d_{\rho}:X\times X \to [0,\infty) \ , \ \ \  d_\rho(x,y):=\sum_{k \in \mathbb{N}} \rho_k d_{k}(x_k,y_k)$$  
for every $x=(x_k)_{k \in \mathbb{N}}, y=(y_k)_{k \in \mathbb{N}}\in X$.

\begin{thm}\label{metricaprod}
	The topology on $X$ generated by $d_\rho$ coincides with $\tau_{prod}$.
	\begin{proof}
		Let $C=\prod_{k\in \mathbb{N}}A_k \in \tau_{prod}$ be an open cylinder such that each $A_k$ is different from the empty set, and take an arbitrary $x=(x_k)_{k\in\mathbb{N}} \in C$. We show that there exists $\epsilon>0$ such that the $d_\rho$-ball of radius $\epsilon$ centered at $x$, i.e., $B_{\epsilon}(x,d_\rho)$, is entirely contained in $C$. Now, for every $k \in \mathbb{N}$, there exists $\epsilon_k>0$ such that $B_{\epsilon_k}(x_k,d_k)\subset A_k$, because $A_k$ belongs to $\tau_k$, which is the topology generated by the metric $d_k$. Let $\epsilon:= \inf \{ \ \rho_k\epsilon_k \ : \ k\in\mathbb{N} \ \text{such that} \ A_k\neq X_k \}$. Note that $\epsilon>0$ because $A_k\neq X_k$ only for a finite number of $k \in \mathbb{N}$. Now, consider a generic $y \in B_{\epsilon}(x,d_\rho)$. For every $k$ such that $A_k \neq X_k$, we have that $\rho_k d_k(x_k,y_k)\leq d_\rho(x,y) <\epsilon \leq \rho_k \epsilon_k$ (by construction). That is, $d_k(x_k,y_k)<\epsilon_k$. In particular $y_k \in A_k$. If instead $k$ is such that $A_k=X_k$, there is nothing to check. Therefore, $B_\epsilon(x,d_\rho)\subset C$.
		
		Now, we show that, given $x \in X$ and $\epsilon > 0$, there exists an open cylinder $C = \prod_{k \in \mathbb{N}} A_k \in \tau_{prod}$ containing $x$ such that $C \subset B_{\epsilon}(x, d_\rho)$. 
		
		Let $\Lambda \in \mathbb{N}$ be such that 
		\[
		\sum_{k > \Lambda}^{\infty} \rho_k < \epsilon/2.
		\] 
		Let $\|\rho\|_{\infty} := \sup_{k \in \mathbb{N}} \rho_k$ and define a family $(A_k)_{k \in \mathbb{N}}$ such that, for every $k \leq \Lambda$, $A_k := B_{\epsilon/(2\|\rho\|_{\infty}\Lambda)}(x_k, d_k)$, and $A_k = X_k$ for all $k > \Lambda$. Clearly, $x \in C := \prod_{k \in \mathbb{N}} A_k$. 
		
		To show that the open cylinder $C$ is entirely contained in $B_{\epsilon}(x, d_\rho)$, consider an arbitrary $y \in C$. By construction, for every $k \leq \Lambda$, we have $d_k(x_k, y_k) < \epsilon / (2\|\rho\|_{\infty}\Lambda)$. It follows that
		\begin{gather*}
		d_\rho(x,y) = \sum_{k \in \mathbb{N}} \rho_k d_{k}(x_k, y_k) = \sum_{k=1}^{\Lambda} \rho_k d_{k}(x_k, y_k) + \sum_{k > \Lambda}^{\infty} \rho_k d_{k}(x_k, y_k) \\
		\leq \|\rho\|_{\infty} \sum_{k=1}^{\Lambda} d_{k}(x_k, y_k) + \sum_{k > \Lambda}^{\infty} \rho_k 
		< \|\rho\|_{\infty} \Lambda \dfrac{\epsilon}{2\|\rho\|_{\infty} \Lambda} + \dfrac{\epsilon}{2} = \epsilon,
		\end{gather*}
		thereby establishing the claim.
	\end{proof}
\end{thm}

\section{Realizations of $\mathbb{T}^\mathbb{N}$ into Sequence Spaces}\label{AAC}

We explore here some specific settings where the infinite torus $\mathbb{T}^\mathbb{N}$ endowed with the product topology arises. To this end, let us first define the sequence spaces that will provide the framework for our constructions. 

For $s \in \mathbb{R}$, we denote by $h^s$ the Sobolev-Hilbert  space of sequences $u=(u_k)_{k \in \mathbb{N}} \in \mathbb{C}^\mathbb{N}$ such that $\sum_{k \in \mathbb{N}} k^{2s} |u_k|^2 < \infty$, endowed with its natural norm topology. 

Furthermore, given a sequence of positive real numbers $\eta = (\eta_k)_{k \in \mathbb{N}}$, we define the weighted space 
\begin{equation*}
\ell^\infty_\eta:=\left\{u=(u_k)_{k \in \mathbb{N}} \in \mathbb{C}^\mathbb{N} \; : \; \sup_{k \in \mathbb{N}}\eta_k|u_k|<\infty \right\},
\end{equation*}
which can be endowed either with the norm topology or with the weak-* topology. 

To define the latter, we introduce the pre-dual space $\ell^1_{1/\eta}$ as the Banach space of sequences $\phi = (\phi_k)_{k \in \mathbb{N}} \in \mathbb{C}^\mathbb{N}$ such that
\begin{equation*}
\|\phi\|_{\ell^1_{1/\eta}} := \sum_{k \in \mathbb{N}} \eta_k^{-1} |\phi_k| < \infty.
\end{equation*}
The weak-* topology on $\ell^\infty_\eta$ is then the coarsest topology such that the linear functionals 
\begin{equation*}
\ell^\infty_\eta \ni u \mapsto \langle u, \phi \rangle_{\ell^\infty_\eta,\ell^{1}_{1/\eta}} := \sum_{k \in \mathbb{N}} u_k \phi_k \in \mathbb{C}
\end{equation*}
are continuous for every $\phi \in \ell^1_{1/\eta}$. 

In what follows, we construct suitable topological embeddings of $\mathbb{T}^\mathbb{N}$ into these spaces, considering three distinct cases:

\begin{enumerate}[label=(\roman*)]
	\item the $h^s$ case with norm topology;
	\item the $\ell^\infty_\eta$ case with weak-* topology; 
	\item the $\ell^\infty_\eta$ case with norm topology.
\end{enumerate}

More specifically, let $\xi=(\xi_k)_{k \in \mathbb{N}}$ be a sequence of non-negative real numbers with $\xi_k > 0$ for infinitely many $k \in \mathbb{N}$. We consider the set
\begin{equation}\label{Txi_def}
\mathcal{T}_{\xi} := \left\{ u=(u_k)_{k \in \mathbb{N}} \in \mathbb{C}^\mathbb{N} \; : \; |u_k|^2 = \xi_k, \ \forall k \in \mathbb{N} \right\},
\end{equation}
under the assumption that $\xi$ satisfies suitable decay conditions such that $\mathcal{T}_\xi \subset h^s$ or $\mathcal{T}_\xi \subset \ell^\infty_\eta$, depending on the case. 

In this appendix, we show that:

\begin{itemize}
	\item In Cases (i) and (ii), $\mathcal{T}_\xi$ is always a realization of $\mathbb{T}^\mathbb{N}$;
	\item In Case (iii), $\mathcal{T}_\xi$ is a realization of $\mathbb{T}^\mathbb{N}$ if and only if $\lim_{k \to \infty} \eta_k^2 \xi_k = 0$. 
\end{itemize}

It is worth mentioning that if the sequence of radii is eventually bounded away from zero on the support of $\xi$, namely if $\inf_{k \in S} \eta_k^2 \xi_k > 0$ where $S = \text{supp}(\xi)$, then $\mathcal{T}_\xi$ would instead inherit the structure of a non-compact Banach manifold modeled on $\ell^\infty$.

To simplify the technical discussion, we observe that it is always possible to restrict our analysis to the following normalized setting:
\begin{equation}\label{normalization}
s=0, \quad \eta_k = 1, \quad \text{and} \quad \xi_k > 0 \quad \forall k \in \mathbb{N}.
\end{equation}
Indeed, the isometric isomorphisms
\begin{align}
h^s \ni u &\mapsto v:=(k^s u_k)_{k \in \mathbb{N}} \in \ell^2, \label{map1} \\
\ell^\infty_\eta \ni u &\mapsto v:=(\eta_k u_k)_{k \in \mathbb{N}} \in \ell^\infty, \label{map2}
\end{align}
are homeomorphisms with respect to the strong topologies and, in the case of \eqref{map2}, also with respect to the weak-* topology. 

\begin{rmk}
	The map \eqref{map2} is a weak-* homeomorphism as it is the adjoint of the isometric isomorphism between the pre-dual spaces $f: \ell^1 \to \ell^1_{1/\eta}$ defined by $(f(\phi))_k = \eta_k \phi_k$ for all $\phi \in \ell^1$ and $k \in \mathbb{N}$. Indeed, for any $u \in \ell^\infty_\eta$ and $\phi \in \ell^1$, the duality pairing satisfies
	\begin{equation*}
	\langle (\eta_k u_k), \phi \rangle_{\ell^\infty, \ell^1} = \sum_{k \in \mathbb{N}} \eta_k u_k \phi_k = \sum_{k \in \mathbb{N}} u_k (\eta_k \phi_k) = \langle u, f(\phi) \rangle_{\ell^\infty_\eta, \ell^1_{1/\eta}}.
	\end{equation*}
	Since $f$ is a topological isomorphism between the pre-duals, its adjoint is a homeomorphism between the dual spaces endowed with their weak-* topologies.
\end{rmk}

Furthermore, let $S := \text{supp}(\xi)$ be the set of indices where $\xi_k > 0$. For any $k \notin S$, the condition $|u_k|^2 = \xi_k = 0$ implies $u_k = 0$. Consequently, the projection onto the coordinates in $S$,
\begin{equation*}
\Pi_S: \mathcal{T}_\xi \to \mathbb{C}^S,
\end{equation*}
is a bijection onto its image and restricts to a homeomorphism in all the topological settings considered. Since $S$ is infinite, we can identify $S$ with $\mathbb{N}$ through a monotonic reindexing, which preserves both the decay properties of the sequences and the respective topologies. By virtue of these identifications, in the following subsections we shall assume the conditions in \eqref{normalization} to hold without loss of generality.

\subsection{Embedding into $\ell^2$}

In this case, we assume $\xi \in \ell^1$, which ensures $\mathcal{T}_\xi \subset \ell^2$. We consider the map $\mathcal{J}: \mathbb{T}^\mathbb{N} \to \ell^2$ defined by
\begin{equation}
\mathcal{J}(\theta) := \left( \sqrt{\xi_k} e^{\im \theta_k} \right)_{k \in \mathbb{N}},
\end{equation}
which is a bijection between the infinite torus $\mathbb{T}^\mathbb{N}$ and $\mathcal{T}_\xi$. To show that $\mathcal{J}$ is a homeomorphism, we compare the product topology on $\mathbb{T}^\mathbb{N}$ with the topology induced by the $\ell^2$-norm on $\mathcal{T}_\xi$. 

Recall that the product topology on $\mathbb{T}^\mathbb{N}$ is metrizable; a suitable metric is given by
\begin{equation}\label{dist_T}
d_{\mathbb{T}^\mathbb{N}}(\theta, \varphi) := \sum_{k \in \mathbb{N}} \xi_k \left| e^{\im \theta_k} - e^{\im \varphi_k} \right|,
\end{equation}
where the weights $\xi_k$ ensure the convergence of the series. On the other hand, the $\ell^2$-distance between $u = \mathcal{J}(\theta)$ and $v = \mathcal{J}(\varphi)$ in $\mathcal{T}_\xi$ is
\begin{equation}\label{dist_l2}
d_{\ell^2}(u, v) = \|u - v\|_{\ell^2} = \sqrt{\sum_{k \in \mathbb{N}} \xi_k \left| e^{\im \theta_k} - e^{\im \varphi_k} \right|^2}.
\end{equation}

We first prove the continuity of $\mathcal{J}$. For any $\epsilon > 0$, we seek $\delta > 0$ such that $d_{\mathbb{T}^\mathbb{N}}(\theta, \varphi) < \delta$ implies $d_{\ell^2}(\mathcal{J}(\theta), \mathcal{J}(\varphi)) < \epsilon$. Since $|e^{\im \theta_k} - e^{\im \varphi_k}| \leq 2$, we have
\begin{equation*}
d_{\ell^2}(\mathcal{J}(\theta), \mathcal{J}(\varphi))^2 = \sum_{k \in \mathbb{N}} \xi_k \left| e^{\im \theta_k} - e^{\im \varphi_k} \right|^2 \leq 2 \sum_{k \in \mathbb{N}} \xi_k \left| e^{\im \theta_k} - e^{\im \varphi_k} \right| = 2 d_{\mathbb{T}^\mathbb{N}}(\theta, \varphi).
\end{equation*}
Thus, it suffices to choose $\delta = \epsilon^2/2$.

Conversely, we show that $\mathcal{J}^{-1}$ is continuous. For a fixed $\epsilon > 0$, let $N \in \mathbb{N}$ be large enough such that $2 \sum_{k > N} \xi_k < \epsilon/2$. Then, for any $\theta, \varphi \in \mathbb{T}^\mathbb{N}$,
\begin{align*}
d_{\mathbb{T}^\mathbb{N}}(\theta, \varphi) &< \sum_{k=1}^{N} \xi_k \left| e^{\im \theta_k} - e^{\im \varphi_k} \right| + \frac{\epsilon}{2} \\
&\leq \sqrt{\sum_{k=1}^{N} \xi_k} \sqrt{\sum_{k=1}^{N} \xi_k \left| e^{\im \theta_k} - e^{\im \varphi_k} \right|^2} + \frac{\epsilon}{2} \\
&\leq \sqrt{\|\xi\|_{\ell^1}} \, d_{\ell^2}(\mathcal{J}(\theta), \mathcal{J}(\varphi)) + \frac{\epsilon}{2},
\end{align*}
where we applied the Cauchy-Schwarz inequality to the first $N$ terms. By choosing $\delta = \epsilon / (2\sqrt{\|\xi\|_{\ell^1}})$, we obtain that $d_{\ell^2}(u, v) < \delta$ implies $d_{\mathbb{T}^\mathbb{N}}(\theta, \varphi) < \epsilon$.

\subsection{Embedding into $\ell^\infty$ with the Weak-* Topology}

We consider the realization of the infinite torus $\mathbb{T}^\mathbb{N}$ as the subset $\mathcal{T}_\xi$ of the Banach space $\ell^\infty$ endowed with the weak-* topology. To ensure that $\mathcal{T}_\xi \subset \ell^\infty$, we assume that the sequence of radii satisfies $\sup_{k \in \mathbb{N}} \xi_k < \infty$. 

First, let us characterize the topology of $\mathbb{T}^\mathbb{N}$, which is defined as the product of circles $\mathbb{T} \cong \mathbb{R}/2\pi\mathbb{Z}$. Since each factor $\mathbb{T}$ is homeomorphic to the unit circle $U(1) \subset \mathbb{C}$ via the map $\theta \mapsto e^{\im \theta}$, the product topology on $\mathbb{T}^\mathbb{N}$ is equivalent to the product topology on $U(1)^\mathbb{N}$. This topology is induced by any metric of the form 
\begin{equation}\label{dist_T_U1}
d_{\mathbb{T}^\mathbb{N}}(\theta, \varphi) = \sum_{k \in \mathbb{N}} \rho_k |e^{\im \theta_k} - e^{\im \varphi_k}|,
\end{equation}
where $(\rho_k)_{k \in \mathbb{N}}$ is a summable sequence of positive weights.

Next, we describe the topology on the target space. The weak-* topology on $\ell^\infty$ is the coarsest topology making the linear functionals $u \mapsto \sum_{k \in \mathbb{N}} u_k \phi_k$ continuous for all $\phi \in \ell^1$. A neighborhood of a point $u \in \ell^\infty$ in this topology is given by a finite intersection of sets of the form
\begin{equation}\label{w_star_neigh}
V_{\phi, \epsilon}(u) := \left\{ v \in \ell^\infty \; : \; \left| \sum_{k \in \mathbb{N}} (u_k - v_k) \phi_k \right| < \epsilon \right\}, \quad \phi \in \ell^1, \ \epsilon > 0.
\end{equation}

While the weak-* topology is not metrizable on the whole $\ell^\infty$, its restriction to any bounded set is metrizable. Since $\mathcal{T}_\xi$ is bounded, we can consider a ball $B_R := \{u \in \ell^\infty : \|u\|_{\ell^\infty} \leq R\}$ with $R > \sqrt{\|\xi\|_{\ell^\infty}}$ and define on it the metric
\begin{equation}\label{dist_wstar_def}
d_{w^*}(u, v) := \sum_{k \in \mathbb{N}} 2^{-k} |u_k - v_k|.
\end{equation}

To show that $d_{w^*}$ induces the weak-* topology on $B_R$, we verify that the two neighborhood systems generate the same topology:

\begin{enumerate}
	\item Any neighborhood $V_{\phi, \epsilon}(u)$ contains a metric ball $B_{d_{w^*}}(u, \delta)$. Indeed, for a fixed $\phi \in \ell^1$, there exists $N \in \mathbb{N}$ such that $\sum_{k > N} |\phi_k| < \epsilon / (4R)$. For any $v \in B_R$ such that $d_{w^*}(u, v) < \delta$, the definition of the metric implies $2^{-k}|u_k - v_k| < \delta$ for all $k \in \mathbb{N}$. In particular, for all $k \leq N$, we have the uniform bound $|u_k - v_k| < 2^N \delta$. By choosing $\delta < \epsilon / (2 \cdot 2^N \|\phi\|_{\ell^1})$, we ensure that
	\begin{equation*}
	\left| \sum_{k \in \mathbb{N}} (u_k - v_k) \phi_k \right| \leq \sum_{k=1}^N |u_k - v_k| |\phi_k| + 2R \sum_{k > N} |\phi_k| < 2^N \delta \sum_{k=1}^N |\phi_k| + \frac{\epsilon}{2} < \epsilon.
	\end{equation*}
	
	\item Any metric ball $B_{d_{w^*}}(u, \epsilon)$ contains a weak-* neighborhood. To see this, we first observe that for any $u, v \in B_R$, the distance between coordinates is uniformly bounded by $|u_k - v_k| \leq 2R$. We then choose an index $N \in \mathbb{N}$ such that
	\begin{equation*}
	\sum_{k = N+1}^{\infty} 2^{-k} (2R) < \frac{\epsilon}{2}.
	\end{equation*}
	Now, let $A = \{e_1, \dots, e_N\} \subset \ell^1$ be the finite set of functionals from the canonical basis, and consider the weak-* neighborhood 
	\begin{equation*}
	U_{A, \epsilon/2}(u) := \left\{ v \in B_R \; : \; |u_k - v_k| < \frac{\epsilon}{2}, \ \forall k=1, \dots, N \right\}.
	\end{equation*}
	For any $v \in U_{A, \epsilon/2}(u)$, the metric distance $d_{w^*}(u, v)$ satisfies:
	\begin{align*}
	d_{w^*}(u, v) &= \sum_{k=1}^N 2^{-k} |u_k - v_k| + \sum_{k = N+1}^{\infty} 2^{-k} |u_k - v_k| \\
	&< \frac{\epsilon}{2} \left( \sum_{k=1}^N 2^{-k} \right) + \sum_{k = N+1}^{\infty} 2^{-k} (2R) < \epsilon.
	\end{align*}
	This proves the inclusion $U_{A, \epsilon/2}(u) \subset B_{d_{w^*}}(u, \epsilon)$.
\end{enumerate}

Finally, we consider the map $\mathcal{J}: \mathbb{T}^\mathbb{N} \to \mathcal{T}_\xi$ defined by the coordinate-wise assignment
\begin{equation*}
\mathcal{J}(\theta) := \left( \sqrt{\xi_k} e^{\im \theta_k} \right)_{k \in \mathbb{N}}.
\end{equation*}
By construction, $\mathcal{J}$ is a bijection. To show it is a homeomorphism, we evaluate the metric $d_{w^*}$ restricted to the set $\mathcal{T}_\xi$. For any two points $u = \mathcal{J}(\theta)$ and $v = \mathcal{J}(\varphi)$ in $\mathcal{T}_\xi$, the distance defined in \eqref{dist_wstar_def} reads
\begin{equation*}
d_{w^*}(u, v) = \sum_{k \in \mathbb{N}} 2^{-k} \left| \sqrt{\xi_k} e^{\im \theta_k} - \sqrt{\xi_k} e^{\im \varphi_k} \right| = \sum_{k \in \mathbb{N}} \left( 2^{-k} \sqrt{\xi_k} \right) \left| e^{\im \theta_k} - e^{\im \varphi_k} \right|.
\end{equation*}
By identifying the weights in the product metric \eqref{dist_T_U1} as $\rho_k := 2^{-k} \sqrt{\xi_k}$, which are strictly positive and summable since $\sup_k \xi_k < \infty$, the map $\mathcal{J}$ becomes an isometry between the metric spaces $(\mathbb{T}^\mathbb{N}, d_{\mathbb{T}^\mathbb{N}})$ and $(\mathcal{T}_\xi, d_{w^*})$. 

Since we have already established that $d_{w^*}$ induces the weak-* topology on bounded sets of $\ell^\infty$, it follows that $\mathcal{J}$ is a homeomorphism between the infinite torus $\mathbb{T}^\mathbb{N}$ and the set $\mathcal{T}_\xi$ endowed with the weak-* topology. 

\subsection{Embedding into $\ell^\infty$ with the Strong Topology}

In this subsection, we consider $\ell^\infty$ endowed with the strong (norm) topology induced by $\|u\|_{\ell^\infty} = \sup_{k \in \mathbb{N}} |u_k|$. We prove that if $\lim_{k \to \infty} \xi_k = 0$, then the map $\mathcal{J}: \mathbb{T}^\mathbb{N} \to \mathcal{T}_\xi$ is a homeomorphism.

To show this, we verify that the product topology on $\mathbb{T}^\mathbb{N}$ and the norm topology on $\mathcal{T}_\xi$ are equivalent through $\mathcal{J}$.

\begin{enumerate}
	\item \textbf{Continuity of $\mathcal{J}$:} Let $\epsilon > 0$ and $\theta \in \mathbb{T}^\mathbb{N}$. Since $\sqrt{\xi_k} \to 0$, there exists an index $N \in \mathbb{N}$ such that $\sqrt{\xi_k} < \epsilon/2$ for all $k > N$. We define a cylinder $U$ in the product topology as:
	\begin{equation*}
	U := \left\{ \varphi \in \mathbb{T}^\mathbb{N} \; : \; |e^{\im \theta_k} - e^{\im \varphi_k}| < \frac{\epsilon}{\max_{j \leq N} \sqrt{\xi_j}}, \quad \forall k = 1, \dots, N \right\}.
	\end{equation*}
	For any $\varphi \in U$, we evaluate the distance $\|\mathcal{J}(\theta) - \mathcal{J}(\varphi)\|_{\ell^\infty}$ by splitting the supremum:
	\begin{itemize}
		\item For $k \leq N$, the definition of $U$ implies:
		\begin{equation*}
		\sqrt{\xi_k} |e^{\im \theta_k} - e^{\im \varphi_k}| < \sqrt{\xi_k} \frac{\epsilon}{\max_{j \leq N} \sqrt{\xi_j}} \leq \epsilon.
		\end{equation*}
		\item For $k > N$, the condition on the radii ensures:
		\begin{equation*}
		\sqrt{\xi_k} |e^{\im \theta_k} - e^{\im \varphi_k}| \leq 2\sqrt{\xi_k} < 2 \cdot \frac{\epsilon}{2} = \epsilon.
		\end{equation*}
	\end{itemize}
	Since the supremum over all $k \in \mathbb{N}$ is less than $\epsilon$, it follows that $\|\mathcal{J}(\theta) - \mathcal{J}(\varphi)\|_{\ell^\infty} < \epsilon$, proving continuity.
	
	\item \textbf{Continuity of $\mathcal{J}^{-1}$:} Let $u = \mathcal{J}(\theta) \in \mathcal{T}_\xi$ and let $C$ be a basic cylinder centered at $\theta$ in $\mathbb{T}^\mathbb{N}$, defined by a finite set of indices $K \subset \mathbb{N}$ and open intervals of width $\delta_k > 0$:
	\begin{equation*}
	C = \{ \varphi \in \mathbb{T}^\mathbb{N} \; : \; |e^{\im \theta_k} - e^{\im \varphi_k}| < \delta_k, \ \forall k \in K \}.
	\end{equation*}
	To prove that $\mathcal{J}^{-1}$ is continuous, we find $\epsilon > 0$ such that $B_{\ell^\infty}(u, \epsilon) \cap \mathcal{T}_\xi \subset \mathcal{J}(C)$. Since $K$ is a finite set and $\sqrt{\xi_k} > 0$ for all $k$, we can define:
	\begin{equation*}
	\epsilon := \min_{k \in K} \left( \sqrt{\xi_k} \delta_k \right) > 0.
	\end{equation*}
	Now, let $v = \mathcal{J}(\varphi) \in \mathcal{T}_\xi$ be any point such that $\|u - v\|_{\ell^\infty} < \epsilon$. By the definition of the supremum norm, for each $k \in K$ we have:
	\begin{equation*}
	\sqrt{\xi_k} |e^{\im \theta_k} - e^{\im \varphi_k}| \leq \sup_{j \in \mathbb{N}} \sqrt{\xi_j} |e^{\im \theta_j} - e^{\im \varphi_j}| = \|u - v\|_{\ell^\infty} < \epsilon.
	\end{equation*}
	Using the definition of $\epsilon$, we observe that $\epsilon \leq \sqrt{\xi_k} \delta_k$ for every $k \in K$. It follows that:
	\begin{equation*}
	\sqrt{\xi_k} |e^{\im \theta_k} - e^{\im \varphi_k}| < \sqrt{\xi_k} \delta_k.
	\end{equation*}
	Dividing by $\sqrt{\xi_k}$, we obtain $|e^{\im \theta_k} - e^{\im \varphi_k}| < \delta_k$ for all $k \in K$, which implies $\varphi \in C$. This confirms that $\mathcal{J}^{-1}$ is continuous.
\end{enumerate}

The condition $\lim_{k \to \infty} \xi_k = 0$ is essential. If it were not satisfied, say $\inf_{k \in \mathbb{N}}\xi_k = c > 0$, a norm-ball of radius $\epsilon < c$ would impose a constraint on an infinite number of angles simultaneously. Such a set cannot contain the image of any cylinder (which only constrains finitely many coordinates), and $\mathcal{T}_\xi$ would behave as a non-compact Banach manifold rather than a compact torus.

\section{Algebraic Background}\label{AAB}

\subsection{Localization of Abelian Groups}
\begin{defn}
	Let $S:=\Z\setminus \{0\}$ and $G$ an abelian group. The localization of $G$ by $S$ is the set of equivalence classes
	\begin{equation*}
	S^{-1}G:=\bigg\{\dfrac{g}{s}, \; g \in G, \; s \in S \; : \; \dfrac{g}{s} = \dfrac{g'}{s'} \Leftrightarrow \; \exists n \in S \; \text{such that} \; n(sg'-s'g)=0 \ \ \text{in} \ G \bigg\}.
	\end{equation*}
\end{defn}

\begin{prop}\label{1}
	$S^{-1}G$ is an abelian group with the operation
	\begin{equation}\label{groupsum}
	\dfrac{g}{s}+\dfrac{h}{r}=\dfrac{rg+sh}{sr}
	\end{equation}
	\begin{proof}
		First, observe that, given $g\in G$ and $s \in S$, for any $m \in S$ we have
		\begin{equation}\label{m}
		\dfrac{g}{s}=\dfrac{mg}{ms}.
		\end{equation}
		Now, let $g,h,g',h'\in G$ and $s,r,s',r'\in S$ be such that
		\begin{equation}
		\dfrac{g}{s}=\dfrac{g'}{s'} \ \ \ \text{and} \ \ \  \dfrac{h}{r}=\dfrac{h'}{r'}.
		\end{equation}
		Then, by definition, there exist $n,m\in S$ such that $n(sg'-s'g)=m(rh'-r'h)=0$. Thus,
		\begin{gather}
		\dfrac{g'}{s'}+\dfrac{h'}{r'}=\dfrac{r'g'+s'h'}{r's'}=\dfrac{nmrsr'g'+nmrss'h'}{nmrsr's'}
		=\dfrac{mr'rnsg'+ns'smrh'}{nmr's'rs}\\=\dfrac{mr'rns'g+ns'smr'h}{nmr's'rs}=\dfrac{rg+sh}{rs}=\dfrac{g}{s}+\dfrac{h}{r}.
		\end{gather}
		So, the group operation \eqref{groupsum} is well defined and, of course, it is associative. The identity element is given by
		\begin{equation}
		\dfrac{0}{1}
		\end{equation}
		and, given $g \in G$, $s \in S$,
		\begin{equation}
		-\dfrac{g}{s}=\dfrac{-g}{s}=\dfrac{g}{-s}.
		\end{equation}
		Finally, the commutativity property of the operation just defined follows trivially.
	\end{proof}
\end{prop}

\begin{defn}
	An abelian group is said to be without torsion if there are no elements $g \in G$, $g \neq 0$, that admit a non-zero integer $n$ such that $ng = 0$.
\end{defn}

\begin{prop}\label{2}
	If $G$ is without torsion, the group homomorphism
	\begin{equation}
	\imath:G \to S^{-1}G, \ \ \ g\mapsto \dfrac{g}{1}
	\end{equation}
	is injective.
	\begin{proof}
		Assume $g/1=0/1$. Then, there exists $n \in S$ such that $n(1\cdot g-1\cdot 0)=ng=0$. The statement follows by the fact that $G$ has not torsion elements.
	\end{proof}
\end{prop}

\begin{prop}\label{3}
	If $G$ is without torsion, $S^{-1}G$ is a $\Q$-vector space.
	\begin{proof}
		Being $S^{-1}G$ an abelian group, and hence a $\Z$-module, it is sufficient to define an appropriate notion of multiplication by rational scalars. For any rational number $\textbf{r} \in \Q$ and any $g\in G$, $s\in S$, we define
		\begin{equation}
		\textbf{r}\dfrac{g}{s}=\dfrac{pg}{qs},
		\end{equation}
		where $p \in \Z$ and $q \in S$ are such that $\textbf{r}=p/q$. By using \eqref{m}, it is very easy to show that this operation is well defined.
	\end{proof}
\end{prop}

\begin{thm}
	Any countable abelian group without torsion can be canonically embedded in a countable vector space over $\Q$. 
	\begin{proof}
		The proof follows by Propositions \ref{1},\ref{2} and \ref{3}, and by the fact that, if $G$ is countable, then $S^{-1}G$ is also countable.
	\end{proof}
\end{thm}

\subsection{Basics of Abelian Groups}

\begin{defn}[Independent Set]
	Let $G$ be an abelian group. Given a set of indices $J$ and a subset $\{g_j\}_{j\in J}$ of $G$, we say that $\{g_j\}_{j \in J}$ is independent if and only if, for any $\nu \in \Z^{(J)}$, $\sum_{j \in J}\nu_jg_j=0$ implies $\nu=0$; otherwise, we say that $\{g_j\}_{j\in J}$ is dependent. 
\end{defn}

\begin{defn}[Maximal Independent Set]
	Let $\{g_j\}_{j \in J}$ be an independent subset of the abelian group $G$. If for any other $g \in G$, $\{g_j\}_{j \in J}\cup \{g\}$ is dependent, then, we say that $\{g_j\}_{j \in J}$ is a maximal independent subset of $G$. If we allow the group to have an infinite number of independent elements, then we have to assume the axiom of choice for maximal independent sets to be well defined.
\end{defn}

\begin{prop}\label{samecard}
	Let $G$ be an abelian group without torsion. Then, all its maximal independent subsets have the same cardinality.
	\begin{proof}
		We show that, for any maximal independent subset $\{g_j\}_{j \in J}\subset G$, $\{g_j/1\}_{j \in J}$ is a basis for $S^{-1}G$. Then, the fact that, for a vector space, any basis has the same cardinality, yields to the conclusion. So, let $\{g_j\}_{j \in J}\subset G$ a maximal independent subset. First, observe that, for any $g/s \in S^{-1}G$, $g\neq 0$, there is a finite family of rational numbers $(r_i)_{i \in I}$, $I\subseteq J$, such that
		\begin{equation}
			\dfrac{g}{s}=\sum_{i \in I}r_i\dfrac{g_i}{1}.
		\end{equation}
		Indeed, let $g \neq 0$. Then, there exist $m \in \N$ and $\nu \in \Z^{(J)}$ such that $mg=\sum_{j \in J}\nu_j g_j$. Therefore,
		\begin{equation}
		\dfrac{g}{s}=\dfrac{mg}{ms}=\dfrac{\sum_{j \in J}\nu_j g_j}{ms}=\sum_{j \in J}\dfrac{\nu_j}{ms}g_j.
		\end{equation}
		Furthermore, let $(p_j/q_j)_{j \in I}$, $I\subseteq J$, be a finite family of rational numbers. Then,
		\begin{equation}
			\sum_{j \in I}\dfrac{p_j}{q_j}\dfrac{g_j}{1}=0 \ \Rightarrow \ \sum_{j \in I}p_j\bigg(\prod_{i \neq j}q_i\bigg) g_j=0 \ \Rightarrow \ p_j=0, \ \forall j \in I.
		\end{equation}
		This implies that $\{g_j/1\}_{j \in J}$ is a basis for the $\Q$-vector space $S^{-1}G$. 
		\end{proof}
\end{prop}

\begin{defn}[Rank]
	The rank of $G$ is the cardinality of any maximal independent subset. 
\end{defn}

\begin{prop}\label{rankofq}
	The rank of $(\Q,+)$, and that of any subgroup of it, is equal to $1$.
	\begin{proof}
		For any $a/b,\ c/d \in \Q$, there exist $m,n \in \Z$, not both vanishing, such that $m(a/b)+n(c/d)=0$. In particular,  if both $a\neq0$ and $c\neq 0$, $m=cb, n=-ad$. 
	\end{proof}
\end{prop}

\begin{prop}\label{rankozj}
	The rank of $\ \Z^{(J)} \ $ equals the cardinality of the set of indices $J$.
	\begin{proof}
		It is sufficient to observe that the subset $\{e_j\}_{j \in J}\subset \Z^{(J)}$, is a maximal independent set.	
	\end{proof}
\end{prop}

\begin{thm}\label{isorank}
	Let $\Phi:G\to G'$ be an isomorphism of abelian groups. Then, the rank of $G$ coincides with the rank of $G'$.
	\begin{proof}
		It is sufficient to prove the followings: (i) if $g,h \in G$ are independent, then $\Phi(g)$ and $\Phi(h)$ are independent; (ii) let $\{g_j\}_{j \in J}$ be a maximal independent subset of $G$, then, for any $g'\in G'$ there exist $m \in \Z$ and $\nu \in \Z^{(J)}$ such that $mg'-\sum_{j \in J}\nu_j\Phi(g_j)=0$. 
		
		Concerning $(i)$, let $n,m \in \Z$ be such that $n\Phi(g)+m\Phi(h)=0$. Then, 
		\begin{equation}
		0=\Phi^{-1}(0)=\Phi^{-1}(n\Phi(g)+m\Phi(h))=\Phi^{-1}(\Phi(ng+mh))=ng+mh.
		\end{equation}
		This implies that $n=m=0$ because, by hypothesis, $g$ and $h$ are independent. 
		
		Now, we prove $(ii)$. Let $g'\in G'$ and $g:=\Phi^{-1}(g')$. Since $\{g_j\}_{j \in J}$ is a maximal independent subset of $G$, there exist $m \in \Z$ and $\nu \in \Z^{(J)}$ such that $mg-\sum_{j \in J}\nu_jg_j=0$. Then, 
		\begin{equation}
			0=\Phi(0)=\Phi\bigg(mg-\sum_{j \in J}\nu_jg_j\bigg)=mg'+\sum_{j \in J}\nu_j\Phi(g_j).
		\end{equation}
		The proof is concluded.
	\end{proof}
\end{thm}

\begin{defn}[Free Abelian Group]
	An abelian group $G$ is said to be free if there exists a maximal independent subset $\{g_j\}_{j \in J}$ of $G$ such that, for any $g \in G$, there exists a unique $\nu \in \Z^{(J)}$ such that $g=\sum_{j \in J}\nu_jg_j$. We call such subset "\textit{basis}".
\end{defn}

\begin{cor}\label{freeZ}
	Let $G$ be a free abelian group and $\{g_j\}_{j \in J}$ a basis. Then $G\cong\Z^{(J)}$.
	\begin{proof}
		By definition, $\Z^{(J)} \ni \nu \mapsto \sum_{j \in J}\nu_jg_j$ is injective and surjective.
	\end{proof}
\end{cor}

\subsection{Subgroups of $(\Q,+)$}

\begin{prop}\label{freesubofq1}
	Let $G$ be a free subgroup of $(\Q,+)$. Then, $G$ is isomorphic to $\Z$.
	\begin{proof}
		By Proposition \ref{rankofq}, $G$ has rank $1$, and the only rank $1$ direct sum of copies of $\Z$ is the direct sum involving only one term (see Proposition \ref{rankozj}). Then, the claim follows by Theorem \ref{isorank}.
	\end{proof}
\end{prop}

\begin{cor}\label{freesubofq2}
	A subgroup $G$ of $(\Q,+)$ is free if and only if there exists $r \in \Q$ such that for any $g \in G$ we have $g=nr$ for some $n \in \Z$.
	\begin{proof}
		Let $\Phi:\Z\to G$ be an isomorphism. Define $r=\Phi(1)$. Then, 
		\begin{equation}
		G=\Phi(\Z)=\{\Phi(n) \; : \; n\in\Z \}=\{n\Phi(1) \; : \; n\in\Z \}=\{nr \; : \; n\in\Z \}.
		\end{equation}
		The claim follows by Proposition \ref{freesubofq1}.
	\end{proof}
\end{cor}

	\begin{proof}[Proof of Lemma \ref{notfree}]\label{notfreeA}
		Assume by contradiction that $G$ is free. By Corollary \ref{freesubofq2}, $G=r\Z$ for some $r \in \Q$. This prevents the existence of such a sequence.
	\end{proof}

\begin{thm}[\cite{baer}\cite{beau}\cite{Fuchs2015}]\label{b}
	Let $\Lambda \in \prod_{j \in \N}\N_0 \cup \{\infty\}$, $(\textbf{p}_j)_{j \in \N}$ the list of all prime numbers in ascending order, and $i$ a natural number that is not divisible by $\textbf{p}_j$ if $\Lambda_j>0$. Then, non-trivial subgroups of the additive rationals are exactly all those groups of the form
	\begin{equation}
	S(i,\Lambda):=\bigg\{\dfrac{n i}{\prod_{j \in \N}\textbf{p}_j^{\lambda_j}} \; : \; \lambda \in \bigoplus_{j \in \N}\N_0 \cap \{0,...,\Lambda_j\}, \ n \in \Z \bigg\} \ ,
	\end{equation}
	Moreover, $S(i,\Lambda)\cong S(i',\Lambda')$ if and only if $\Lambda_j=\Lambda'_j$ for almost all $j$, and, whenever they are different, both are finite.
\end{thm}

\begin{proof}[Proof of Lemma \ref{modsol}]\label{modsolA}
	With the same notations of Theorem \ref{b}, if $\Lambda \in \prod_{j \in \N}\N_0 \cup \{\infty\}$ satisfies $\prod_{j\in \N}\textbf{p}^{\Lambda_j}<\infty$, then $S(i,\Lambda)=i(\prod_{j\in \N} \textbf{p}^{\Lambda_j})^{-1}\Z$. 
	
	So, let us focus on the case $\prod_{j\in\N} \textbf{p}^{\Lambda_j}=\infty$. Let  $(b_k)_{k=2}^\infty$ be the sequence
	\begin{equation}
	b_k=\begin{cases}
	\textbf{p}_\ell \ \ \ \text{if} \ k=\textbf{p}_\ell^{\lambda_\ell}, \ \ \ \text{for some} \ \ell\in\N, \ \lambda_\ell \in \N, \  \lambda_\ell\leq \Lambda_\ell, \\
	0 \ \ \ \text{otherwise}
	\end{cases}
	\end{equation}
	and construct $\textit{\textbf{a}}=(a_j)_{j\in\N}$ by setting $a_1:=1$, $a_j:=b_{k_j}$ for all $j \geq 2$, where $(k_j)_{j\in \N}$ is the support of the sequence $(b_k)_{k\in\N}$, with $k_j<k_{j'}$ if $j<j'$. Now, given $\lambda \in \bigoplus_{j\in\N}\N_0 \cap \{0,...,\Lambda_j\}$, let $N$ be the natural number such that $k_N=\textbf{p}_{j^*}^{\lambda_{j^*}}$,  with $j^*$ such that $\textbf{p}_{j^*}^{\lambda_{j^*}}\geq \textbf{p}_j^{\lambda_j}$ for all $j \in \N$. Then, for any $n \in \Z$, 
	\begin{equation}\label{1f}
	\dfrac{n}{\prod_{j=1}^\infty \textbf{p}_j^{\lambda_j}}=\dfrac{n\prod_{j=1}^\infty\textbf{p}_j^{\overline{\lambda}_j-\lambda_j}}{\prod_{j=1}^Na_j},
	\end{equation}
	where, for all $j \in \N$, $\overline{\lambda}_j:=\max\{\lambda \in \{0,...,\Lambda_j\} \ | \ \textbf{p}_j^{\lambda}\leq \textbf{p}_{j^*}^{\lambda_{j^*}} \}$. On the other hand, given $m \in \Z$ and $N \in \N$, there exists a unique $\lambda \in \bigoplus_{j\in\N}\N_0 \cap \{0,1,...,\Lambda_j\}$ such that
	\begin{equation}\label{2f}
	\dfrac{m}{\prod_{j=1}^Na_j}=\frac{m}{\prod_{j=1}^\infty \textbf{p}_j^{\lambda_j}}.
	\end{equation}
	Clearly, \eqref{1f} and \eqref{2f} are inverses of each other. Moreover, for all $i \in \N$, the map
	\begin{equation}
	Q(\textit{\textbf{a}}) \ni \dfrac{m}{\prod_{j=1}^Na_j} \mapsto \frac{mi}{\prod_{j=1}^\infty \textbf{p}_j^{\lambda_j}} \in S(i,\Lambda),
	\end{equation}
	where $\lambda$ is defined in \eqref{2f}, is a homomorphism of groups. Therefore, for all $i \in \N$, $S(i,\Lambda)$ is isomorphic to $Q(\textit{\textbf{a}})$, with $\textit{\textbf{a}}\in\Sigma$ defined as above.
	
	Conversely, given $\textit{\textbf{a}}=(a_k)_{k \in \N}\in \Sigma$, for all $j \in \N$, let $(\Lambda_j^{(k)})_{k \in \N}$ be the list of natural numbers defined by
	\begin{equation}\label{ak}
	a_k=\prod_{j=1}^\infty \textbf{p}_j^{\Lambda_j^{(k)}}.
	\end{equation}
	We define $\Lambda=(\Lambda_j)_{j\in\N}\in \bigoplus_{j \in \N} \N_0 \cup \{\infty\}$ as
	\begin{equation}\label{lambdaj}
	\Lambda_j:=\sum_{k=1}^{\infty}\Lambda_j^{(k)}.
	\end{equation}
	Note that
	\begin{equation}
	\prod_{k=1}^Na_k=\prod_{k=1}^N\prod_{j=1}^\infty \textbf{p}_j^{\Lambda_j^{(k)}}=\prod_{j=1}^\infty \textbf{p}_j^{\sum_{k=1}^N\Lambda_j^{(k)}},
	\end{equation}
	which implies that, for all $i \in \N$, $Q(\textit{\textbf{a}})$ is isomorphic to $S(i,\Lambda)$, with $\Lambda$ as in \eqref{lambdaj}, via the map
	\begin{equation}
	Q(\textbf{a}) \ni \frac{m}{\prod_{k=1}^Na_k} \mapsto \frac{n}{\prod_{j=1}^\infty\textbf{p}_j^{\lambda_j}} \in S(i,\Lambda), \quad n=m, \quad \lambda_j=\sum_{k=1}^N\Lambda_j^{(k)}.
	\end{equation}
	This ends the proof.
\end{proof}

\nocite{*}

\end{document}